\documentclass[12pt,a4paper]{amsart}
\usepackage{amssymb}
\usepackage[all]{xy}
\SelectTips{eu}{}
\usepackage{dsfont}
\usepackage{enumerate}

\calclayout
\numberwithin{figure}{section}
\numberwithin{equation}{section}
\newtheorem{theorem}[equation]{Theorem}
\newtheorem{lemma}[equation]{Lemma}
\newtheorem{proposition}[equation]{Proposition}
\newtheorem{corollary}[equation]{Corollary}

\theoremstyle{definition}
\newtheorem{defn}[equation]{Definition}
\newtheorem{hyp}[equation]{Hypothesis}
\newtheorem{remark}[equation]{Remark}

\theoremstyle{remark}
\newtheorem*{Claim}{Claim}

\newcommand{\one}{\mathds 1}

\newcommand{\bsr}{\boldsymbol{r}}
\newcommand{\bss}{\boldsymbol{s}}

\newcommand{\cl}{\operatorname{cl}}
\newcommand{\col}{\colon}
\newcommand{\cone}{\operatorname{cone}}

\newcommand{\da}{{\downarrow}}
\newcommand{\End}{\operatorname{End}}
\newcommand{\Ext}{\operatorname{Ext}}

\newcommand{\funct}{\operatorname{\mathcal{H}\!\!\;\mathit{om}}}
\newcommand{\ges}{{\scriptscriptstyle\geqslant}}
\newcommand{\Hom}{\operatorname{Hom}}
\newcommand{\ik}{ik}

\newcommand{\Inj}{\operatorname{\mathsf{Inj}}}
\newcommand{\Ker}{\operatorname{Ker}}
\newcommand{\kos}[2]{{#1/\!\!/#2}}
\newcommand{\KInj}[1]{\mathsf K(\Inj #1)}
\newcommand{\KacInj}[1]{\mathsf K_{\mathsf{ac}}(\Inj #1)}
\newcommand{\KInjc}[1]{\mathsf K^{\mathsf c}(\Inj #1)}

\newcommand{\Loc}{\operatorname{\mathsf{Loc}}}
\newcommand{\lotimes}{\otimes^{\mathbf L}}
\newcommand{\mmod}{\operatorname{\mathsf{mod}}}
\newcommand{\Mod}{\operatorname{\mathsf{Mod}}}
\newcommand{\nat}[1]{{#1}^{\natural}}
\newcommand{\Proj}{\operatorname{Proj}}
\newcommand{\proj}{\operatorname{\mathsf{proj}}}
\newcommand{\res}{\operatorname{res}}
\newcommand{\RHom}{\operatorname{{\mathbf R}\mathrm{Hom}}}

\newcommand{\Spec}{\operatorname{Spec}}
\newcommand{\stmod}{\operatorname{\mathsf{stmod}}}
\newcommand{\StMod}{\operatorname{\mathsf{StMod}}}
\newcommand{\supp}{\operatorname{supp}}
\newcommand{\Thick}{\operatorname{\mathsf{Thick}}}
\newcommand{\ua}{{\uparrow}}
\newcommand{\xra}{\xrightarrow}

\renewcommand{\Gamma}{\varGamma}
\newcommand{\vf}{\varphi}

\newcommand{\dg}{{dg}~}

\newcommand{\Z}{\mathbb Z}

\newcommand{\mcV}{\mathcal V}
\newcommand{\mcW}{\mathcal W}
\newcommand{\mcZ}{\mathcal Z}

\newcommand{\fa}{\mathfrak a}
\newcommand{\fb}{\mathfrak b}

\newcommand{\fm}{\mathfrak m}
\newcommand{\fp}{\mathfrak p}
\newcommand{\fq}{\mathfrak q}

\newcommand{\sfC}{\mathsf C}
\newcommand{\sfD}{\mathsf D}
\newcommand{\sfc}{\mathsf c}
\newcommand{\sfK}{\mathsf K}
\newcommand{\sfS}{\mathsf S}
\newcommand{\sfT}{\mathsf T}

\title[Stratifying modular representations]
{Stratifying modular representations of\\ finite groups}
\author{Dave Benson}
\address{Dave Benson \\
Department of Mathematics\\
University of Aberdeen\\
Aberdeen AB24 3UE\\
Scotland, U.K.}

\author{Srikanth B. Iyengar}
\address{Srikanth B. Iyengar\\ Department of Mathematics\\
University of Nebraska\\ Lincoln NE 68588\\ U.S.A.}

\author{Henning Krause}
\address{Henning Krause\\ Institut f\"ur Mathematik\\
  Universit\"at Paderborn\\ 33095 Paderborn\\ Germany}
\curraddr{Fakult\"at f\"ur Mathematik\\ Universit\"at Bielefeld\\
  33501 Bielefeld\\ Germany.}



\begin{document}
\begin{abstract}
We classify localising subcategories of the stable module category of
a finite group that are closed under tensor product with simple (or,
equivalently all) modules.  One application is a proof of the
telescope conjecture in this context. Others include new proofs of the
tensor product theorem and of the classification of thick
subcategories of the finitely generated modules which avoid the use of
cyclic shifted subgroups. Along the way we establish similar
classifications for differential graded modules over graded polynomial
rings, and over graded exterior algebras.
\end{abstract}

\maketitle

\setcounter{tocdepth}{1}
\tableofcontents

\section{Introduction}

Let $G$ be a finite group and $k$ a field of characteristic $p$, where $p$ divides the order of $G$.
Let $\Mod(kG)$ be the category of possibly infinite dimensional modules over the group algebra $kG$.
In this article, a full subcategory $\sfC$ of $\Mod(kG)$ is said to be \emph{thick} if it satisfies the following  conditions.
\begin{itemize}
\item 
Any direct summand of a module in $\sfC$ is also in $\sfC$.
\item If
$0 \to M_1 \to M_2 \to M_3 \to 0$
is an exact sequence of $kG$-modules, and two of $M_1$, $M_2$, $M_3$ are in $\sfC$ then so is the third.
\end{itemize}
A thick subcategory $\sfC$ is \emph{localising} when, in addition, the following property holds:
\begin{itemize}
\item If $\{M_\alpha\}$ is a set of modules in $\sfC$ then
$\bigoplus_{\alpha}M_\alpha$ is in $\sfC$.
\end{itemize}
By a version of the Eilenberg swindle the first condition follows from the others.

There is a notion of \emph{support} in this context, introduced by Benson, Carlson and Rickard \cite{Benson/Carlson/Rickard:1996a}. It associates to each $kG$-module $M$ a subset $\mcV_G(M)$ of the set $\Proj H^*(G,k)$ of homogeneous prime ideals in the cohomology ring $H^*(G,k)$ other than the maximal ideal of positive degree elements.

Our main result is that $\Proj H^*(G,k)$ stratifies $\Mod(kG)$, in the following sense.

\begin{theorem}
\label{th:main}
There is a natural one to one correspondence between non-zero localising subcategories of $\Mod(kG)$ that are closed under tensoring with simple $kG$-modules and subsets of $\Proj H^*(G,k)$.

The localising subcategory corresponding to a subset $\mcV\subseteq\Proj H^*(G,k)$ 
is the full subcategory of modules $M$ satisfying $\mcV_G(M)\subseteq\mcV$.
\end{theorem}

A more precise version of this result is given in Theorem~\ref{th:Mod}. It is 
modelled on Neeman's classification \cite{Neeman:1992a} of the localising subcategories of the 
unbounded derived category $\sfD(\Mod R)$ of complexes of modules over a noetherian 
commutative ring $R$. Neeman's work in turn was inspired by Hopkins' classification \cite{Hopkins:1987a} of the thick subcategories of the perfect complexes over $R$ in terms of specialisation closed subsets of $\Spec R$. The corresponding classification problem for the stable category $\stmod(kG)$ of finitely generated $kG$-modules was solved by Benson, Carlson and Rickard \cite{Benson/Carlson/Rickard:1997a}, at least in the case where $k$ is algebraically closed. There is also a discussion in that paper as to why one demands
that the subcategories are closed under tensor products with simple modules. For a $p$-group this condition is automatically satisfied, but for an arbitrary finite group there is a subvariety of $\Proj H^*(G,k)$ called the \emph{nucleus}, that encapsulates the obstruction to a classification of all localising subcategories, at least for the principal block.

Applications of Theorem~\ref{th:main} include a classification of
smashing localisations of the stable category $\StMod(kG)$ of all
$kG$-modules, a proof of the telescope conjecture in this setting, a
classification of localising subcategories that are closed under
products and duality, and a description of the left perpendicular category of a
localising subcategory.

We also provide new proofs of the subgroup theorem, the tensor product 
theorem and the classification of thick subcategories of $\stmod(kG)$. 
The proofs of these results given 
in \cite{Benson/Carlson/Rickard:1996a,Benson/Carlson/Rickard:1997a} 
rely heavily on the use of cyclic shifted subgroups and an infinite 
dimensional version of Dade's lemma, which play no role in our work.

As intermediate steps in the proof of the classification theorems for 
modules over $kG$, we establish analogous results on differential 
graded modules over polynomial rings, and over exterior algebras. 
These are of independent interest.

This paper is the third in a sequence devoted to supports and localising 
subcategories of triangulated categories. We have tried to make this paper 
as easy as possible to read without having to go through the first two papers 
\cite{Benson/Iyengar/Krause:2008a,Benson/Iyengar/Krause:bik2} in the series. 
In particular, some of the arguments in them have been repeated in this more 
restricted context for convenience.

\subsection*{Acknowledgments}
The work of the second author was partly supported by NSF grant DMS 0602498. The first and second authors are grateful to the Humboldt Foundation for providing support which enabled them to make extended research visits to Paderborn to work with the third author. All three authors are grateful to the Mathematische Forschungsinstitut at Oberwolfach for support through a ``Research in Pairs'' visit. The authors also thank the referee for useful comments.

\section{Strategy}\label{se:strategy}

The proof of Theorem~\ref{th:main} consists of a long chain of 
transitions from one category to another. In this section, we give 
an outline of the strategy.

The first step is to reduce to the stable module category $\StMod(kG)$.
This category has the same objects as the module category $\Mod(kG)$, 
but the morphisms in $\StMod(kG)$ are given by quotienting out those 
morphisms in $\Mod(kG)$ that factor through a projective module. 
The category $\StMod(kG)$ is a triangulated category, in which the 
triangles come from the short exact sequences of $kG$-modules. See 
for example Theorem I.2.6 of Happel \cite{Happel:1988a}
or \S5 of Carlson \cite{Carlson:1996a} for further details.

A \emph{thick subcategory} of a triangulated category is a full
triangulated subcategory that is closed under direct summands. A
\emph{localising subcategory} of a triangulated category is a
full triangulated subcategory that is closed under direct sums. By an
Eilenberg swindle, localising subcategories are also closed under
direct summands, and hence thick.

Recall that, given $kG$-modules $M$ and $N$, one considers $M\otimes_{k}N$ as a $kG$-module with the diagonal $G$-action. The $kG$-modules $M\otimes_{k}N$ and $N\otimes_{k}M$ are isomorphic, and $P\otimes_{k}N$ is projective for any projective $kG$-module $P$. Thus the tensor product on $\Mod(kG)$ passes down to a tensor product on $\StMod(kG)$. In either context, we say that a thick, or localising, subcategory $\sfC$ is \emph{tensor ideal} if it satisfies the following condition.
\begin{itemize}
\item If $M$ is in $\sfC$ and $S$ is a simple $kG$-module then $M \otimes_k S$
is in $\sfC$.
\end{itemize}
This condition is vacuous if $G$ is a finite $p$-group since the only simple module is $k$ with trivial $G$-action. Furthermore, every $kG$-module has a finite filtration whose subquotients are direct sums of simple modules (induced by the radical filtration of $kG$), so the condition above is equivalent to:
\begin{itemize}
\item If $M$ is in $\sfC$ and $N$ is any $kG$-module then $M\otimes_k N$ is in $\sfC$.
\end{itemize}

\begin{proposition}
\label{pr:Mod-to-StMod}
Every non-zero tensor ideal localising subcategory of $\Mod(kG)$ contains the localising subcategory of projective modules.

The canonical functor from $\Mod(kG)$ to $\StMod(kG)$ induces a one to
one correspondence between non-zero tensor ideal localising
subcategories of $\Mod(kG)$ and tensor ideal localising subcategories of
$\StMod(kG)$.
\end{proposition}

\begin{proof}
The tensor product of any module with $kG$ is a direct sum of
copies of $kG$. So if $\sfC$ is a non-zero tensor ideal localising subcategory
of $\Mod(kG)$,  it contains $kG$,
and hence every projective $kG$-module. This proves the first statement
of the proposition. The rest is now clear.
\end{proof}

There are some technical disadvantages to working in $\StMod(kG)$ that
are solved by moving to a slightly larger triangulated category:
$\KInj{kG}$, the category whose objects are the complexes of injective
$kG$-modules and whose morphisms are the homotopy classes of degree
preserving maps of complexes. The tensor product of modules extends to
complexes, and defines a tensor product on $\KInj{kG}$. This category
was investigated in detail by Benson and Krause
\cite{Benson/Krause:2008a}.  Taking the Tate resolution $tM=M\otimes_k
tk$ of a $kG$-module $M$ gives an equivalence of triangulated
categories from the stable module category $\StMod(kG)$ to the full
subcategory $\KacInj{kG}$ of $\KInj{kG}$ consisting of acyclic
complexes.  This equivalence preserves the tensor product. It suffices
to check that $tk\otimes_k tk\cong tk$, and an explicit isomorphism
can be found in Section~XII.4 of Cartan and Eilenberg
\cite{Cartan/Eilenberg:1956}.

The Verdier quotient of $\KInj{kG}$ by $\KacInj{kG}$ is the unbounded derived category $\sfD(\Mod kG)$. There are left and right adjoints, forming a recollement
\[ 
\StMod(kG) \xrightarrow{\sim} \KacInj{kG}
 \ \begin{smallmatrix} \Hom_k(tk,-) \\
\hbox to 50pt{\leftarrowfill} \\ \hbox to 50pt{\rightarrowfill} \\ 
\hbox to 50pt{\leftarrowfill} \\ - \otimes_k tk
\end{smallmatrix} \ \KInj{kG} \ \begin{smallmatrix} \Hom_k(pk,-) \\
\hbox to 50pt{\leftarrowfill} \\ \hbox to 50pt{\rightarrowfill} \\ 
\hbox to 50pt{\leftarrowfill} \\ - \otimes_k pk
\end{smallmatrix} \ \sfD(\Mod kG) 
\]
where $pk$ denotes a projective resolution of $k$.

It is shown in \cite{Benson/Krause:2008a} that the theory of supports
for $\StMod(kG)$ developed in \cite{Benson/Carlson/Rickard:1996a}
extends in a natural way to $\KInj{kG}$. Exactly one more prime ideal
appears in the theory, namely the maximal ideal $\fm$ of positive
degree elements in $H^*(G,k)$. We write $\Spec H^*(G,k)$ for the set
$\Proj H^*(G,k)\cup\{\fm\}$ of all homogeneous prime ideals in
$H^*(G,k)$. If $X$ is an object in $\KInj{kG}$ then there is
associated to it a \emph{support} $\mcV_G(X)$, which is a subset of
$\Spec H^*(G,k)$; see \S\ref{se:strat}.

\begin{proposition}
\label{pr:StMod-to-KInj}
For every tensor ideal localising subcategory $\sfC$ of $\StMod(kG)$ there
are two tensor ideal localising subcategories of $\KInj{kG}$. One is the image
of $\sfC$ and the other is generated by this together with $pk$.
This sets up a two to one correspondence between tensor ideal localising 
subcategories of $\KInj{kG}$ and those of $\StMod(kG)$.
\end{proposition}

\begin{proof}
First we claim that for any non-zero object $X$ in $\sfD(\Mod kG)$,
the tensor ideal localising subcategory generated by $X$ is the whole
of $\sfD(\Mod kG)$.

Indeed, for any non-zero $kG$-module $M$ the $kG$-module $M\otimes_{k}kG$ is free and non-zero. Therefore, 
$H^{*}(X\otimes_{k}kG)$ is a non-zero direct sum of copies of $kG$, since it is isomorphic to $H^{*}(X)\otimes_{k}kG$. This implies that in $\sfD(\Mod kG)$ the complex $X\otimes_{k}kG$ is isomorphic 
to a non-zero direct sum of copies of shifts of $kG$. Hence $kG$ is in the tensor ideal localising subcategory generated by $X$. It remains to note that every object in $\sfD(\Mod kG)$ is in the localising subcategory generated by $kG$.

Next observe that the full subcategory $\KacInj{kG}$ of $\KInj{kG}$
consisting of acyclic complexes is a tensor ideal localising
subcategory. Thus the image in $\KInj{kG}$ of every tensor ideal
localising subcategory of $\StMod(kG)$ is a tensor ideal
localising subcategory.

Now let $\sfD$ be a tensor ideal localising subcategory of $\KInj{kG}$
that is not in the image of $\StMod(kG)$. Then $\sfD$ contains some
object $X$ which is not acyclic. The image of $X$ in $\sfD(\Mod kG)$
is non-zero, and hence generates $\sfD(\Mod kG)$. Thus the tensor
ideal containing $X$ also contains $pk$. It follows that $\sfD$ is
generated by $pk$ and its intersection with the image of $\StMod(kG)$.
\end{proof}

A version of the main theorem for $\KInj{kG}$ is as follows; see also
Theorem~\ref{th:classify-KInj}.

\begin{theorem}
\label{th:loc}
There is a natural one to one correspondence between tensor ideal localising 
subcategories $\sfC$ of $\KInj{kG}$ and subsets of $\Spec H^*(G,k)$.

The localising subcategory corresponding to a subset $\mcV$ of $\Spec H^*(G,k)$
is the full subcategory of complexes $X$ satisfying $\mcV_G(X)\subseteq \mcV$.
\end{theorem}

The proof of Theorems~\ref{th:main} and \ref{th:loc} occupies
\S\S\ref{se:poly}--\ref{se:main-theorem}.  An outline of the proof is
as follows. Using an appropriate version of the Quillen stratification
theorem, it suffices to work with an elementary abelian $p$-group
\[ 
E=\langle g_1,\dots,g_r\rangle .
\] 
Since $k$ is of characteristic $p$ the group algebra $kE$ is isomorphic to
the algebra
\[ 
k[z_1,\dots,z_r]/(z_1^p,\dots,z_r^p),
\]
where the (image of the element) $z_i$ corresponds to $g_i-1$.

We write $A$ for the Koszul complex on $kE$ with respect to
$z_1,\dots,z_r$, and view it as a differential graded (dg)
algebra. Thus, as a graded algebra $A$ is the exterior algebra over
$kE$ on indeterminates $y_{1},\dots,y_{r}$, with each $y_i$ of degree
$-1$, and the differential on $A$ is determined by setting
\[ 
d(z_i)=0 \qquad\text{and}\qquad d(y_i)=z_i. 
\]
Observe that the elements $z_1^{p-1}y_1,\dots,z_r^{p-1}y_r$ are cycles in $A$ of degree $-1$. Let $\Lambda$ be an exterior algebra over $k$  on $r$ generators $\xi_i$ in degree $-1$, and view it as a \dg algebra with zero differential. The map $\Lambda \to A$ given by 
\[ 
\xi_i \mapsto z_i^{p-1}y_i 
\] 
is a morphism of \dg algebras and a quasi-isomorphism; see \S\ref{se:Lambda-to-A}.

Let $S$ be a graded polynomial ring $k[x_1,\dots,x_r]$ where each variable $x_i$ is of degree $2$, and view it as a \dg algebra with zero differential. One has an isomorphism of graded $k$-algebras $S\cong \Ext^{*}_{\Lambda}(k,k)$; see \S\ref{se:bgg}.

Let $\KInj{A}$ and $\KInj{\Lambda}$ denote the homotopy categories of graded-injective \dg modules over $A$ and  $\Lambda$ respectively; see \S\ref{se:gr-inj}. Our strategy is to establish first a classification of the localising subcategories of $\sfD(S)$, the derived category of \dg $S$-modules, and then successively for $\KInj{\Lambda}$,  $\KInj{A}$, $\KInj{kE}$, $\KInj{kG}$, and finally for $\StMod(kG)$ and $\Mod(kG)$. The following diagram provides an overview of the proof of Theorem~\ref{th:main}.
\begin{gather*}
\overset{\S\ref{se:poly}}{\sfD(S)} 
\overset{\S\ref{se:bgg}}{\leadsto} 
\KInj{\Lambda} \overset{\S\ref{se:Lambda-to-A}}{\leadsto} 
\KInj{A} \overset{\S\ref{se:A-to-kE}}{\leadsto} 
\KInj{kE} \overset{\S\ref{se:kE-to-kG}}{\leadsto} \KInj{kG}
\overset{\S\ref{se:main-theorem}}{\leadsto} \Mod(kG).  \\
\text{\centerline{Leitfaden}}
\end{gather*}

The passage from $S$ to $kE$ is modelled on the work of Avramov, Buchweitz, Iyengar, and Miller \cite{Avramov/Buchweitz/Iyengar/Miller:hffc}, where it is used to establish results on numerical invariants of complexes over commutative local rings by tracking them along a chain of categories as above. The focus here is on tracking structural information. A crucial idea in executing this passage from $S$ to $kG$ is that of a
\emph{stratification} of a tensor triangulated category, which allows one to focus on minimal localising subcategories. In \S\ref{se:strat} we describe the general theory of stratifications, and show that when
a tensor triangulated category is stratified, one can classify its tensor ideal localising subcategories. This development is partly inspired by the work of Hovey, Palmieri, and Strickland \cite{Hovey/Palmieri/Strickland:1997a}.

In \S\ref{se:applics} we describe various applications, including a
classification of smashing localisations of $\StMod(kG)$, and new
proofs of the subgroup theorem, the tensor product theorem, and the
classification of thick subcategories of $\stmod(kG)$.

\section{Stratifications}\label{se:strat}

In this section, we recall 
from  \cite{Benson/Iyengar/Krause:2008a,Benson/Iyengar/Krause:bik2} 
those parts of the general theory of support varieties and stratifications 
that we wish to use in this paper.

Let $\sfT$ be a triangulated category admitting arbitrary
coproducts. Usually, we denote $\Sigma$ the shift on $\sfT$. Given a
subcategory $\sfC$ of $\sfT$ we write $\Loc(\sfC)$ for the localising
subcategory generated by $\sfC$; this is the smallest localising
subcategory of $\sfT$ containing $\sfC$. The thick subcategory
generated by $\sfC$ is denoted $\Thick(\sfC)$. An object $C$ of $\sfT$
is a \emph{generator} if $\Loc(C)=\sfT$. The standing assumption in
this article is that $\sfT$ is generated by a single compact object.
Recall that an object $C$ is \emph{compact} if the functor
$\Hom_{\sfT}(C,-)$ commutes with coproducts.

For objects $X,Y$ in $\sfT$ we write $\Hom^*_\sfT(X,Y)$ for the 
graded abelian group with degree $n$ component 
$\Hom_\sfT(X,\Sigma^n Y)$, and $\End^{*}_{\sfT}(X)$ for the graded 
ring $\Hom^{*}_{\sfT}(X,X)$.

Let $R$ be a graded commutative noetherian ring; thus $R$ is a
$\mathbb Z$-graded noetherian ring such that $rs= (-1)^{|r||s|}sr$ for
all homogeneous elements $r,s$ in $R$. We say that the triangulated
category $\sfT$ is \emph{$R$-linear}, or that $R$ \emph{acts on}
$\sfT$, if we are given a homomorphism of graded rings $\phi\col R\to
Z^*\sfT$, to the graded centre of $\sfT$. This means that for each
object $X$ there is a homomorphism of graded rings
\[
\phi_{X}\col R \to \End^*_\sfT(X)
\]
such that for each pair of objects $X,Y$ the induced left and right 
actions of $R$ on $\Hom^*_\sfT(X,Y)$ agree up to the usual sign; 
see \cite[\S4]{Benson/Iyengar/Krause:2008a}.

Let $\sfT$ be an $R$-linear triangulated category. 

We write $\Spec R$ for the set of homogeneous prime ideals in $R$. 
For any ideal $\fa$ in $R$ the subset $\{\fp\in\Spec R\mid \fp\supseteq \fa\}$ 
is denoted $\mcV(\fa)$. A subset $\mcV\subseteq\Spec R$ is 
\emph{specialisation closed} if $\fp\in\mcV$ and $\fq\supseteq\fp$ 
imply $\fq\in\mcV$. Given such a subset $\mcV$,  pick a compact generator 
$C$ of $\sfT$, and let $\sfT_\mcV$ be the full subcategory 
\[ 
\sfT_\mcV=\{X \in \sfT \mid 
\Hom^*_\sfT(C,X)_{\fp}=0\ \forall\,\fp\in\Spec R\setminus\mcV\}. 
\]
This is a localising subcategory of $\sfT$. The following result is proved in 
\cite[\S4]{Benson/Iyengar/Krause:2008a}.

\begin{proposition}
\label{pr:indep-of-C}
The localising subcategory $\sfT_\mcV$ depends only on $\mcV$ and not on the choice of compact generator $C$. 
Furthermore, there is a localisation functor $L_\mcV\col\sfT\to\sfT$ such that $L_\mcV X=0$ if and 
only if $X\in\sfT_\mcV$. \qed
\end{proposition}

This result and the theory of Bousfield localisation imply
that there is an exact functor $\Gamma_\mcV\col\sfT\to\sfT$ and for
each object $X$ in $\sfT$ an exact \emph{localisation triangle}
\begin{equation}
\label{eq:localisation}
\Gamma_\mcV X \to X \to L_\mcV X \to. 
\end{equation}

Proposition~\ref{pr:indep-of-C} has the following useful consequence.

\begin{corollary}
\label{cor:2actions}
Let $\phi,\phi'\col R \to Z^*\sfT$ be actions of $R$ on $\sfT$. If
there exists a compact generator $C$ for $\sfT$ for which the maps
$\phi_C,\phi_C'\col R \to \End_\sfT^*(C)$ agree, then for any
specialisation closed set $\mcV\subseteq \Spec R$ the functors
$\Gamma_{\mcV}$ and $L_\mcV$ defined through $\phi$ agree with those
defined through $\phi'$.
\end{corollary}

\begin{proof}
The definitions of $\Gamma_\mcV$ and $L_\mcV$ only depend on the action of $R$ on $\Hom^*_\sfT(C,X)$, since $C$ is a compact generator. The action factors through the map $R \to\End_\sfT^*(C)$, which justifies the claim. 
\end{proof}

In this work the principal source of an $R$-action on $\sfT$ is a tensor structure on it. 

\subsection*{Tensor triangulated categories} \ \medskip

\noindent
Let $(\sfT,\otimes,\one)$ be a \emph{tensor triangulated category}. By this we mean the following structure; see \cite[\S8]{Benson/Iyengar/Krause:2008a} for details.
\begin{itemize}
\item $\sfT$ is a compactly generated triangulated category with coproducts.  
\item $\sfT$ is symmetric monoidal, with \emph{tensor product} $\otimes\col\sfT\times\sfT\to\sfT$ and \emph{unit} $\one$.
\item The tensor product is exact in each variable and preserves coproducts.
\item The unit $\one$ is compact.
\item Compact objects are strongly dualisable.
\end{itemize}
In this article we make also the following assumption, for simplicity of exposition.
\begin{itemize}
\item $\sfT$ has a single compact generator.
\end{itemize}
We write $\Loc^{\otimes}(\sfC)$ for the tensor ideal localising subcategory 
generated by a subcategory $\sfC$ of $\sfT$. Observe that $\Loc(\sfC)\subseteq\Loc^{\otimes}(\sfC)$ and that the two categories coincide when the unit $\one$ is a generator for $\sfT$; this is the case in many of our contexts.

The stable category $\StMod(kG)$ of a finite group $G$ is tensor triangulated, where the tensor product is $M\otimes_{k}N$ with diagonal $G$-action and the unit is $k$. The homotopy category $\KInj{kG}$ is also tensor triangulated with the same tensor product, but here the unit is $\ik$, the injective resolution of $k$. In either case, the unit generates the category when $G$ is a $p$-group; see \cite{Benson/Krause:2008a} for details.

The ring $\End^{*}_{\sfT}(\one)$ is graded commutative and acts on
$\sfT$ via homomorphisms
\[ 
\End^{*}_{\sfT}(\one)\xrightarrow{X \otimes -} 
\End^{*}_{\sfT}(X),
\]
for each $X$ in $\sfT$. Thus any homomorphism of graded rings $R\to\End^{*}_{\sfT}(\one)$ induces an $R$-action on $\sfT$, and we call such an $R$-action \emph{canonical}.

In the remainder of this section $\sfT$ denotes a tensor triangulated category and $R$ a graded commutative noetherian ring that acts canonically on $\sfT$. 

For each specialisation closed subset $\mcV$ of $\Spec R$ there are
natural isomorphisms
\[
\Gamma_{\mcV}X \cong \Gamma_{\mcV} \one\otimes X\quad\text{and}\quad
L_{\mcV}X \cong L_{\mcV} \one\otimes X.
\]
Thus, the localisation triangle \eqref{eq:localisation} is obtained by applying $-\otimes X$ to the triangle
\[
\Gamma_\mcV \one \to \one \to L_\mcV \one \to. 
\]
These are analogues of Rickard's triangles \cite{Rickard:1997a} in $\StMod(kG)$; see \cite[\S\S4,10]{Benson/Iyengar/Krause:2008a}.

Next we recall the construction of Koszul objects from
\cite[Definition 5.10]{Benson/Iyengar/Krause:2008a}. Given a homogeneous element $r$ of $R$
and an object $X$ in $\sfT$, we write $\kos Xr$ for any object that
appears in an exact triangle
\[
X\xrightarrow{\ r } \Sigma^{|r|}X \to \kos Xr\to.
\]
Iterating this construction on a sequence of elements
$\bsr=r_{1},\dots,r_{n}$ one gets an object in $\sfT$ that is denoted
$\kos X{\bsr}$. Given an ideal $\fa$ in $R$, we write $\kos X{\fa}$
for any Koszul object on a finite sequence of elements generating
$\fa$.

While a Koszul object $\kos X{\fa}$ depends on a choice of a
generating sequence for $\fa$, the thick subcategory it generates does
not. This follows from the second part of the following statement,
where $\sqrt{\fa}$ denotes the radical of $\fa$ in $R$.

\begin{lemma}
\label{le:koszul-pr}
Let $\fa$ be an ideal in $R$ and $X$ an object in $\sfT$. 
\begin{enumerate}[\quad\rm(1)]
\item $\kos X{\fa}$ is in $\Thick(\Gamma_{\mcV(\fa)}X)$.
\item $\kos X{\fa}$ is in $\Thick(\kos X{\fb})$ for any ideal $\fb$
with $\sqrt{\fb}\subseteq \sqrt{\fa}$.
\end{enumerate}
\end{lemma}

\begin{proof}
 By construction it is clear that $\kos X{\fa}$ is in $\Thick(X)$. Since the functor $\Gamma_{\mcV(\fa)}$ is exact, this implies that $\Gamma_{\mcV(\fa)}(\kos X{\fa})$ is in $\Thick(\Gamma_{\mcV(\fa)}X)$. It remains to note that  $\Gamma_{\mcV(\fa)}(\kos X{\fa})\cong \kos X{\fa}$, for $\kos X{\fa}$ is in $\sfT_{\mcV(\fa)}$ by \cite[Corollary 5.11]{Benson/Iyengar/Krause:2008a}. This proves (1).

The claim in (2) can be proved along the same lines as \cite[Lemma 6.0.9]{Hovey/Palmieri/Strickland:1997a}.
\end{proof}

The second part of the next result improves upon \cite[Theorem 6.4]{Benson/Iyengar/Krause:2008a}.

\begin{proposition}
\label{pr:loc-pr}
Let $\mcV$ be a specialisation closed subset of $\Spec R$.
\begin{enumerate}[\quad\rm(1)]
\item If $\mcW\supseteq \mcV$ is specialisation closed, then $\Gamma_{\mcV}X$ is in $\Loc^{\otimes}(\Gamma_{\mcW}X)$.
\item Let $C$ be a compact generator of $\sfT$. For any decomposition $\mcV=\bigcup_{i\in I}\mcV(\fa_{i})$, where each $\fa_{i}$ is an ideal in $R$, there are equalities
\[
\sfT_{\mcV}= \Loc(\{\kos C{\fa_{i}}\mid i\in I\}) = \Loc(\{\Gamma_{\mcV(\fa_{i})}C\mid i\in I\}). 
\]
\end{enumerate}
\end{proposition}

\begin{proof} 
(1) {}From \cite[Proposition 6.1]{Benson/Iyengar/Krause:2008a} one gets the first isomorphism below:
\[
\Gamma_{\mcV}X \cong \Gamma_{\mcV}(\Gamma_{\mcW}X)\cong \Gamma_{\mcV}\one\otimes\Gamma_{\mcW}X.
\]
Thus $\Gamma_{\mcV}X$ is in $\Loc^{\otimes}(\Gamma_{\mcW}X)$.

(2) For each $\fp\in\mcV$ there exists an $i$ in $I$ such that $\fp\in \mcV(\fa_{i})$ holds, so that $\kos C{\fp}$ is in $\Thick(\kos C{\fa_{i}})$, by Lemma~\ref{le:koszul-pr}(2). This fact and \cite[Theorem 6.4]{Benson/Iyengar/Krause:2008a}
imply the last of the following inclusions:
\[
\Loc(\{\kos C{\fa_{i}}\mid i\in I\}) \subseteq
\Loc(\{\Gamma_{\mcV(\fa_{i})}C\mid i\in I\}) \subseteq 
\sfT_{\mcV} \subseteq \Loc(\{\kos C{\fa_{i}}\mid i\in I\}).
\]
The first inclusion holds by Lemma~\ref{le:koszul-pr}(1). The second one holds as $\Gamma_{\mcV(\fa_{i})}C$ is in $\sfT_{\mcV}$ for each $i$, since $\sfT_{\mcV(\fa_{i})}\subseteq \sfT_{\mcV}$. The inclusions above yield the desired result.
\end{proof}

\subsection*{Support} \ \medskip

\noindent
Fix a point $\fp\in\Spec R$ and let $\mcV$ and $\mcW$ be specialisation closed sets of primes such that $\mcV\setminus\mcW=\{\fp\}$. It is shown in \cite[Theorem 6.2]{Benson/Iyengar/Krause:2008a} that the functor $\Gamma_\fp=\Gamma_\mcV L_\mcW$ is independent of choice of $\mcV$ and $\mcW$ with these properties, and coincides with $L_\mcW\Gamma_\mcV$. There are natural isomorphisms
\[
\Gamma_\fp X \cong \Gamma_\fp\one\otimes X.
\]
Following \cite[\S5]{Benson/Iyengar/Krause:2008a}, we define the \emph{support} of an object $X$ in $\sfT$ to be
\[ 
\supp_RX=\{\fp\in\Spec R\mid\Gamma_\fp X \ne 0\}. 
\]
If the $R$-module $\Hom^{*}_{\sfT}(C,X)$ is finitely generated for
some compact generator $C$ for $\sfT$ and $\fa$ is its annihilator
ideal, then $\supp_{R}X$ equals the Zariski-closed subset $\mcV(\fa)$;
see \cite[Theorem 5.5(1)]{Benson/Iyengar/Krause:2008a}.

The notion of support defined above is an analogue of the one for modular representations of a finite group $G$ given in Benson, Carlson and Rickard \cite{Benson/Carlson/Rickard:1996a}. Indeed, if $\sfT=\StMod(kG)$ and $\fp$ is a non-maximal prime in $H^*(G,k)$, then $\Gamma_\fp\one$ is the kappa module for $\fp$ defined in \cite{Benson/Carlson/Rickard:1996a}; see \cite[\S10]{Benson/Iyengar/Krause:2008a}.

The following theorem is the ``local-global principle'', studied in
\cite{Benson/Iyengar/Krause:bik2} for general triangulated categories.

\begin{theorem}
\label{th:local-global}
Let $X$ be an object in $\sfT$. Then
\[ 
X \in \Loc^\otimes(\{\Gamma_\fp X \mid \fp \in \supp_RX\}). 
\]
\end{theorem}

\begin{proof}
It suffices to prove that $\one$ is in the subcategory
\[ 
\sfC=\Loc^\otimes\{\Gamma_\fp\one\mid\fp\in\Spec R\}. 
\]
Indeed, one then gets the desired result by tensoring with $X$, keeping in mind that if $\fp\not\in\supp_RX$ then $\Gamma_\fp\one\otimes X=\Gamma_\fp X=0$.

Proposition~\ref{pr:loc-pr}(1) implies that the set
\[ 
\mcV=\{\fp\in\Spec R \mid \Gamma_{\mcV(\fp)}\one\in\sfC\}. 
\]
is specialisation closed. We claim that $\mcV=\Spec R$. If not, then
since the ring $R$ is noetherian there exists a prime $\fq$ maximal in
$\Spec R\setminus \mcV$. Setting $\mcZ=\mcV(\fq)\setminus\{\fq\}$ one
gets an exact triangle
\[ 
\Gamma_{\mcZ}\one \to \Gamma_{\mcV(\fq)}\one \to \Gamma_\fq\one\to. 
\]
Since $\mcZ\subseteq\mcV$, it follows from
Proposition~\ref{pr:loc-pr}(2) that
\begin{align*} 
\Gamma_{\mcZ}\one\in\sfT_\mcZ\subseteq\sfT_\mcV
=\Loc\{\Gamma_{\mcV(\fp)}C\mid
\fp\in\mcV\}\subseteq\Loc^\otimes\{\Gamma_{\mcV(\fp)}\one\mid
\fp\in\mcV\}\subseteq\sfC,
\end{align*}
where $C$ denotes any compact generator of $\sfT$.  By definition,
$\Gamma_\fq\one$ is in $\sfC$ so it follows from the triangle above that
$\Gamma_{\mcV(\fq)}\one$ is also in it. This contradicts the choice of
$\fq$. Therefore $\mcV=\Spec R$ holds. It remains to note that
\[
\one = \Gamma_{\Spec R}\one \in \sfC
\]
where the equality holds because $\Gamma_{\Spec R}$ is the identity on $\sfT$, since $\sfT_{\Spec R}=\sfT$.
\end{proof}

\subsection*{Stratification} \ \medskip

\noindent
Let $\sfT$ be an $R$-linear tensor triangulated category. For each $\fp$ in $\Spec R$, the full subcategory $\Gamma_\fp\sfT$ is tensor ideal and localising. It consists of objects $X$ in $\sfT$ with the property that the $R$-module $\Hom^{*}_{\sfT}(C,X)$ is $\fp$-local and $\fp$-torsion, where $C$ is some compact generator for $\sfT$; see \cite[Corollary 4.10]{Benson/Iyengar/Krause:2008a}. 

We say that $\sfT$ is \emph{stratified by} $R$ if $\Gamma_\fp\sfT$ is either zero or minimal among tensor ideal localising subcategories for each $\fp$ in $\Spec R$. 

Given a localising subcategory $\sfC$ of $\sfT$ and a subset $\mcV$ of $\Spec R$, we write 
\begin{align*}
\sigma(\sfC)&=\supp_R\sfC=\{\fp\in\Spec R\mid\Gamma_\fp\sfC\ne 0\}\\
\tau(\mcV) &=\{X\in\sfT\mid \supp_RX\subseteq \mcV\}. 
\end{align*}
It follows from \cite[\S8]{Benson/Iyengar/Krause:2008a} that the subcategory $\tau(\mcV)$ is tensor ideal and localising, so these are maps
\begin{equation}
\label{eq:classification} 
\xymatrix{\{\text{\rm tensor ideal localising subcategories of $\sfT$}\}
\ar@<.5ex>[r]^(.64)\sigma &
\{\text{subsets of $\supp_R\sfT$}\}.\ar@<.5ex>[l]^(.36)\tau}
\end{equation}

The theorem below is the reason stratifications are relevant to this work.

\begin{theorem}
\label{th:transfer-strat}
If $\sfT$ is stratified by $R$, then $\sigma$ and $\tau$ are mutually inverse bijections 
between the tensor ideal localising subcategories of $\sfT$ and subsets of $\supp_R\sfT$. 
\end{theorem}

\begin{proof}
It is clear that $\sigma\tau(\mcV)=\mcV$ for any subset $\mcV$ of $\supp_R\sfT$ and that $\sfC\subseteq\tau\sigma(\sfC)$ for any tensor ideal localising subcategory $\sfC$. It remains to check that $\tau\sigma(\sfC)\subseteq\sfC$. 

For each $\fp\in\supp_R \sfC$, minimality of $\Gamma_{\fp}\sfT$ implies the equality below
\[
\Gamma_\fp\sfT=\Gamma_\fp\sfC\subseteq \sfC
\]
while the inclusion holds because $\Gamma_{\fp}X\cong \Gamma_{\fp}\one\otimes X$. For any $X$ in $\tau\sigma(\sfC)$ one has $\supp_RX\subseteq\sigma(\sfC)$, so Theorem~\ref{th:local-global} and the inclusions above imply $X\in\sfC$.
\end{proof}

The following characterisation of minimality is often useful.

\begin{lemma}
\label{le:loc-minimal}
Let $\sfT$ be a tensor triangulated category and $\sfC$ a  non-zero tensor ideal localising subcategory of $\sfT$. The following statements are equivalent:
\begin{enumerate}[\quad\rm(a)]
\item The tensor ideal localising subcategory $\sfC$ is minimal. 
\item For all non-zero objects $X$ and $Y$ in $\sfC$ there exists an object $Z$ in $\sfT$ such that $\Hom^*_\sfT(X\otimes Z,Y)\ne 0$ holds.
\item If $C$ is a compact generator for $\sfT$, then all non-zero objects $X$ and $Y$ in $\sfC$ satisfy $\Hom^*_\sfT(X\otimes C,Y)\ne 0$.
\end{enumerate}
\end{lemma}

\begin{proof}
(a) $\Rightarrow$ (b): Let $X$ be a non-zero object in
$\sfC$. Minimality of $\sfC$ implies that $\Loc^{\otimes}(X)=\sfC$.
Therefore, if there exists $Y$ in $\sfC$ such that for all $Z$ in
$\sfT$ we have $\Hom^*_\sfT(X\otimes Z,Y)=0$, then in particular 
$\Hom^*_\sfT(Y,Y)=0$ and hence $Y=0$.

(b) $\Leftrightarrow$ (c): This is clear, since $\Loc(C)=\sfT$ for any generator $C$.

(c) $\Rightarrow$ (a): Suppose that $\sfC$ is not minimal, so that it
contains a non-zero proper tensor ideal localising subcategory, say
$\sfC'$. Let $X$ be a non-zero object in $\sfC'$. It follows from
\cite[Corollary 4.4.3]{Neeman:2001a} that there is a localisation
functor with kernel $\Loc(X\otimes C)$, so that for each object $W$ in
$\sfT$ there is a triangle $W'\to W\to W''\to$ with $W'\in\sfC'$ and
$\Hom_\sfT^*(X\otimes C,W'')=0$. Pick an object $W$ in $\sfC$ but not
in $\sfC'$ and set $Y=W''$. Since $W'$ and $W$ are in $\sfC$, so is
$Y$ and $W$ is not in $\sfC'$ one gets $Y\ne 0$. Finally, we have
$\Hom^*_\sfT(X\otimes C,Y)=0$, a contradiction.
\end{proof}

We wish to transfer stratifications from one tensor triangulated
category to another. In particular, we need to change the tensor
product on a fixed triangulated category.  It turns out to be
inconvenient to have to keep track of the map from $R$ into the graded
centre of $\sfT$. So we formulate the following principle.

\begin{lemma}
\label{le:2actions}
Let $\sfT$ be a triangulated category admitting two tensor
triangulated structures with the same unit object, $\one$, and let
$\phi,\phi'\col R \to Z^*\sfT$ be two actions.  If $\one$ generates
$\sfT$ and the maps $\phi_\one,\phi_\one'\col R \to \End_\sfT^*(\one)$
agree, then $\sfT$ is stratified by $R$ through $\phi$ if  it is
stratified by $R$ through $\phi'$.
\end{lemma}

\begin{proof}
It follows from Corollary~\ref{cor:2actions} that $\Gamma_\fp X$
defined through $\phi$ and $\phi'$ agree. This justifies the claim,
since localising subcategories are tensor ideal.
\end{proof}

We require also a change of rings result for stratifications.

\begin{lemma}
\label{le:2rings}
Let $\sfT$ be an $R$-linear tensor triangulated category and $\phi\col Q\to R$ a homomorphism of rings.
If $\sfT$ is stratified by the induced action of $Q$, then it is stratified by $R$.
\end{lemma}

\begin{proof}
Fix a prime $\fp$ in $\Spec R$ and set $\fq = \phi^{-1}(\fp)$. It is straightforward to verify that if an object $X$ in $\sfT$ is $\fp$-torsion and $\fp$-local for the $R$-action, then it is $\fq$-torsion and $\fq$-local, for the $Q$-action. The claim is now obvious.
\end{proof}

\section{Graded-injective \dg modules}
\label{se:gr-inj}

This section concerns a certain homotopy category of \dg modules over
a \dg algebra. The development is based on the work of Avramov, Foxby,
and Halperin~\cite{Avramov/Foxby/Halperin:dga}, which is also our
general reference for this material. The main results we prove here,
Theorems~\ref{th:kinj-generator} and \ref{th:strat-descent}, play a
critical role in \S\S\ref{se:bgg}--\ref{se:A-to-kE} and are
tailored for ready use in them; they are not the best one can do in
that direction.

\begin{hyp}
\label{hyp:A}
Let $A$ be a \dg algebra over a field $k$ with the following properties.
\begin{enumerate}[\quad\rm(1)]
\item $A^i=0$ for $i>0$ and $A$ is finite dimensional over $k$.
\item $A^{0}$ is a local ring with residue field $k$.
\item $A$ has a structure of a cocommutative \dg Hopf $k$-algebra.
\end{enumerate}
\end{hyp}

For the definition of a \dg Hopf algebra see \cite[\S21]{Felix/Halperin/Thomas:2001a}. The main consequence of the Hopf structure used in our work is that there is an isomorphism of \dg $A$-modules
\[
\Hom_{k}(A,k)\cong \Sigma^{d} A\quad\text{for some integer $d$.}
\]
This isomorphism can be verified as in
\cite[\S3.1]{Benson:1991a}. Note that $A^{0}$ is also a cocommutative
Hopf algebra and the inclusion $A^{0}\subseteq A$ is compatible with
the Hopf structure. In particular, $A$ is free as a graded module over
$A^{0}$; see \cite[\S3.3]{Benson:1991a}.

We write $\nat A$ for the graded algebra underlying $A$. When $M$ is a
\dg $A$-module, $\nat M$ is the underlying graded $\nat
A$-module. Following \cite{Avramov/Foxby/Halperin:dga}, we say that a
\dg $A$-module $I$ is \emph{graded-injective} if $\nat I$ is an
injective object in the category of graded $\nat A$-modules. Under
Hypothesis~\ref{hyp:A} this condition is equivalent to saying that
$\nat I$ is a graded free $\nat A$-module. We write $\KInj A$ for the
homotopy category of graded-injective \dg $A$-modules. The objects of
this category are graded-injective \dg $A$-modules and morphisms
between such \dg modules are identified if they are homotopic. The
category $\KInj A$ is given the usual structure of a triangulated
category, where the exact triangles correspond to exact sequences of
graded-injective \dg modules. This is analogous to the case of the
homotopy category of complexes over a ring, as described for instance
in Verdier's article in SGA4$\scriptstyle\frac{1}{2}$
\cite{Deligne:1977a}; see also \cite{Verdier:1996a}.

\subsection*{$\KInj A$ is compactly generated}\ \medskip

\noindent
A \dg $A$-module $I$ is said to be \emph{semi-injective} if it is
graded-injective and the functor $\Hom_{\sfK}(-,I)$, where $\sfK$ is
the homotopy category of \dg $A$-modules, takes quasi-isomorphisms to
isomorphisms. Every \dg module $X$ admits a \emph{semi-injective
resolution}: a quasi-isomorphism $X \to I$ of \dg $A$-modules with $I$
semi-injective.

\begin{lemma}
\label{le:injectives-filtration}
Let $A$ be a \dg algebra satisfying Hypothesis~\emph{\ref{hyp:A}}. Each graded-injective \dg $A$-module $I$
has a family $\{I(n)\}_{n\in\Z}$ of \dg submodules with the properties below.
\begin{enumerate}[\quad\rm(1)]
\item For each integer $n$, one has $I(n-1)\subseteq I(n)$. 
\item For each integer $i$, there exists an $n_i$ with $I(n_i)^{\ges i}=I^{\ges i}$.
\item For each integer $n$, the \dg module $I(n)$ is semi-injective.
\end{enumerate}
When $I^{i}=0$ for $i\ll 0$ the \dg $A$-module $I$ is semi-injective.
\end{lemma}

\begin{proof}
Hypothesis~\ref{hyp:A} implies that the $\nat A$-module $\nat I$ is free, so $\nat I = \bigoplus_{j\in\Z} \nat A U^j$, where $U^j$ is the set of basis elements in degree $j$. Set  
\[
I(n)=\bigoplus_{j\ge-n}A U^j.
\]
Since $A^{i}=0$ for $i>0$, by Hypothesis~\ref{hyp:A}, it follows that $d(U^{j})\subseteq I(j-1)$ holds for each $j$, and hence each $I(n)$ is a \dg $A$-submodule of $I$. Conditions (1) and (2) are immediate by construction, and it is evident that each $I(n)$ is graded-injective with $I(n)^{i}=0$ for $i\ll 0$.
It thus remains to verify the last claim of the lemma, for that would also imply that the $I(n)$ are semi-injective.

Assume $I^{i}=0$ for $i\ll 0$ and let $I(n)$ be as above. There are canonical surjections $I\to I/I(n)$ for each $n$ and $I\cong \lim_{n} I/I(n)$ as \dg modules over $A$. It thus suffices to prove that each $I/I(n)$ is semi-injective. So we may assume that $\nat I = \bigoplus_{j\in\Z} \nat A U^j$ with $U^{j}=\varnothing$ for $|j|\gg 0$. By induction on the number of non-empty $U^{j}$, it suffices to consider the case when $I$ is  a free \dg $A$-module. Thus $I$ has the form $A\otimes_{k}V$ with $V$ a graded $k$-vector space with zero differential. The self-duality of $A$ then implies that $I$ is isomorphic to a shift of $\Hom_{k}(A,V)$  and hence semi-injective. 
\end{proof}

The statement below is a special case of a result from
\cite{Avramov/Foxby/Halperin:dga}. We thank the authors for allowing
us to reproduce the proof here.

\begin{lemma}
\label{le:injectives-structure}
Let $A$ be a \dg algebra satisfying Hypothesis~\emph{\ref{hyp:A}}, and let $\fm$ be the dg ideal $\Ker(A\to k)$. Each graded-injective \dg $A$-module $I$ is isomorphic in $\KInj A$ to a graded-injective \dg $A$-module $J$ whose differential satisfies $d(J)\subseteq \fm J$.
\end{lemma}

\begin{proof}
We repeatedly use the fact that graded-injective \dg $A$-modules are graded-free. For any \dg module $F$ we write $\cone(F)$ for the mapping cone of the identity map $F\xra{=}F$. When $F$ is graded-free, $\cone(F)$ has the following lifting property: If $\alpha \col M\to N$ is a surjective morphism of \dg $A$-modules, then for any morphism $\beta\col\cone(F)\to N$ there is a morphism $\gamma\col\cone(F)\to M$ such that $\alpha\gamma=\beta$. This is readily verified by a diagram chase. This observation is used below.

Since $A/\fm A$ is the field $k$, the complex of $k$-vector spaces $I/\fm I$ is isomorphic to $H^{*}(I/\fm I)\oplus \cone(V)$, where $V$ is a graded $k$-vector space with zero differential. Pick a surjective morphism $F\to V$ of \dg $A$-modules with $F$ a free \dg $A$-module and $F/\fm F\cong V$; here $V$ is viewed as a \dg $A$-module via the morphism $A\to k$. One thus gets a surjective morphism $\cone(F)\to\cone(V)$ of \dg $A$-modules. The composed morphism $\cone(F)\to \cone(V) \to I/\fm I$ then lifts to a morphism $\gamma\col \cone(F)\to I$ of \dg $A$-modules. It follows by construction that the map $\gamma\otimes_{A}k$ is injective, which implies that $\gamma$ itself is split-injective. Hence its cokernel, say, $J$, is graded-free. It is also not hard to verify that $d(J)\subseteq \fm J$ holds. Finally, $\cone(F)\cong 0$ in $\KInj A$ and hence $I$ is homotopy equivalent to $J$. 
\end{proof}

The result below is an analogue of \cite[Proposition~2.3]{Krause:2005a} for \dg algebras. In it $\KInjc A$ denotes the full subcategory of compact objects in $\KInj A$. As usual, $\sfD(A)$ stands for the derived category of \dg modules over $A$. We write $\sfD^{\mathsf f}(A)$ for the full subcategory whose objects are \dg modules $M$ such that the $H^*(A)$-module $H^*(M)$ is finitely generated; equivalently, $H^*(M)$ is finite dimensional over $k$.
 
\begin{theorem}
\label{th:kinj-generator}
Let $A$ be a \dg algebra satisfying Hypothesis~\emph{\ref{hyp:A}} and
$\ik$ a semi-injective resolution of the \dg $A$-module $k$. The triangulated
category $\KInj A$ is compactly generated by $\ik$ and the canonical
functor $\KInj A\to \sfD(A)$ restricts to an equivalence
\[ 
\KInjc A\xrightarrow{\, \sim \,}\sfD^{\mathsf f}(A). 
\]
In particular,  $\Thick(\ik)=\KInjc A$ and $\Loc(\ik)=\KInj A$. 
\end{theorem}

\begin{proof}
The \dg $A$-module $k$ has a semi-injective resolution $I$ with
$I^{j}=0$ for $j<0$. One way to construct it is to start with a semi-free
resolution $F\to k$ with $F^{i}=0$ for $i>0$ and apply
$\Hom_{k}(-,k)$; note that $\Hom_{k}(F,k)$ is semi-injective, by
adjunction. Semi-injective resolutions of $\ik$ are isomorphic in
$\KInj A$ so we may assume that $\ik$ is concentrated in non-negative
degrees.

We have to prove that $\ik$ is compact in $\KInj A$ and that it
generates $\KInj A$.  Let $\sfK$ denote the homotopy category of \dg
$A$-modules with the usual triangulated structure. Identifying $\KInj
A$ with a full subcategory of $\sfK$ one gets an identification of $\Hom_{\KInj
A}(-,-)$ with $\Hom_{\sfK}(-,-)$.

We claim that for any graded-injective module $I$, the natural map $k\to\ik$ induces an isomorphism $\Hom_{\sfK}(\ik,I)\cong \Hom_{\sfK}(k,I)$. Indeed, the mapping cone of the canonical inclusion $k\to\ik$ gives rise to an exact triangle 
$k\to\ik\to C \to$ in $\sfK$ with $C^{j}=0$ for $j< 0$. The desired result is that 
\[ 
\Hom_{\sfK}(C,I)=0=\Hom_{\sfK}(\Sigma^{-1}C,I). 
\] 
By Lemma~\ref{le:injectives-filtration}, there exists a semi-injective \dg module $J\subseteq I$ with $J^{\ges -2}=I^{\ges -2}$. One then has isomorphisms
\[ 
0\cong \Hom_{\sfK}(\Sigma^{n}C,J) \cong \Hom_{\sfK}(\Sigma^{n}C,I)
\quad\text{for each $n\leq 0$,} 
\]
where the second one holds for degree reasons, and the first one
because $H^*(C)=0$ and $J$ is semi-injective.  This proves the claim.

If $\{I_{\alpha}\}$ is a set of graded-injective \dg modules over $A$ then the claim  yields
the first and third isomorphisms below:
\begin{align*}
\Hom_{\sfK}\big(\ik,\bigoplus_{\alpha}I_{\alpha}\big)
        &\cong \Hom_{\sfK}\big(k,\bigoplus_{\alpha}I_{\alpha}\big) \\
        &\cong \bigoplus_{\alpha}\Hom_{\sfK}(k,I_{\alpha})\\
        &\cong \bigoplus_{\alpha}\Hom_{\sfK}(\ik,I_{\alpha}).
\end{align*}
The second isomorphism holds because the $\nat A$-module $k$ is finitely generated. Therefore the \dg module $\ik$ is compact in $\KInj A$.

Suppose $I$ is a graded-injective \dg module with
$\Hom^*_{\sfK}(\ik,I)=0$. We wish to verify that $I$ is homotopy
equivalent to $0$. We may assume that $d(I)\subseteq \fm I$, by
Lemma~\ref{le:injectives-structure}, and hence that the differential
on the \dg module $\Hom_{A}(k,I)$ is zero.  This explains the first
isomorphism below:
\[ 
\Hom_{A}(k,I) \cong H^*(\Hom_{A}(k,I)) \cong \Hom^*_{\sfK}(k,I)\cong
\Hom^*_{\sfK}(\ik,I)=0.
\]
The second isomorphism is standard and the third one holds by the
claim established above. The equality is by hypothesis, and it follows that $I=0$.
\end{proof}
 
The following test for equivalence of triangulated categories is implicit in \cite[\S4.2]{Keller:1994a}.  The proof uses a standard d\'evissage argument. Recall that $\sfT^{\mathsf c}$ denotes the subcategory of compact objects in $\sfT$.

\begin{lemma}
\label{le:equivalence}
Let $F\col\sfS \to \sfT$ be an exact functor between compactly
generated triangulated categories with coproducts. If $F$ preserves coproducts
and restricts to an equivalence
$\sfS^{\sfc}\xrightarrow{\sim} \sfT^{\sfc}$, then $F$ is an
equivalence of categories.

In particular, if there exists a compact generator $C$ of $\sfS$ such
that $F(C)$ is a compact generator of $\sfT$ and the induced map
$\End_\sfS^*(C)\to\End_\sfT^*(FC)$ is an isomorphism, then $F$ is an
equivalence.
\end{lemma}

\begin{proof}
Fix a compact object $D$ of $\sfS$ and let $\sfS_D$ be the full
subcategory with objects $X$ in $\sfS$ for which the induced map
$F_{D,X}\col\Hom_\sfS(D,X)\to\Hom_\sfT(FD,FX)$ is a bijection. This is
a localising subcategory and contains $\sfS^{\sfc}$ by the assumption
on $F$.  Therefore $\sfS_D=\sfS$. Given this, a similar argument shows
that for any object $Y$ in $\sfS$ the subcategory $\{X\in\sfS\mid
F_{X,Y}\text{ is bijective}\}$ equals $\sfS$. Thus $F$ is fully
faithful. The essential image of $F$ is a localising subcategory of
$\sfT$ and contains a set of compact generators. We conclude that $F$
is an equivalence.
\end{proof}

This is used in the proof of the following result.
 
\begin{proposition}
\label{pr:kinj-quism}
Let $\vf\col A\to B$ be a morphism of \dg $k$-algebras where $A$ and
$B$ satisfy Hypothesis~\emph{\ref{hyp:A}}. One has an exact functor
$\Hom_{A}(B,-)\col\KInj A\to\KInj B$ of triangulated categories.  If
$\vf$ is a quasi-isomorphism, then this functor is an equivalence and
sends a semi-injective resolution of $k$ over $A$ to a semi-injective
resolution of $k$ over $B$.
\end{proposition}

\begin{proof}
When $I$ is a graded-injective \dg module over $A$, the adjunction isomorphism 
\[ 
\Hom_B(-,\Hom_{A}(B,I)) \cong \Hom_{A}(-,I) 
\]
implies that $\Hom_{A}(B,I)$ is a graded-injective over $B$.  Since
$\Hom_{A}(B,-)$ is additive it defines an exact functor at the level
of homotopy categories. The isomorphism above also implies that when
$I$ preserves quasi-isomorphisms so does $\Hom_{A}(-,I)$. Hence, when
$I$ is semi-injective so is $\Hom_{A}(B,I)$.
 
Suppose $\vf$ is a quasi-isomorphism and let $\ik$ be a semi-injective
resolution of $k$ over $A$. The \dg $B$-module $\Hom_A(B,\ik)$ is then
semi-injective and has cohomology $k$. It is hence a semi-injective
resolution of $k$ over $B$. In view of Theorem~\ref{th:kinj-generator}
one thus gets the following commutative diagram.
\[ 
\xymatrixrowsep{2pc}
\xymatrixcolsep{4.5pc}
\xymatrix{
\KInjc A \ar@{->}[r]^{\Hom_A(B,-)} \ar@{->}[d]^{\sim} & \KInjc B 
\ar@{->}[d]^{\sim} \\
\sfD^{\mathsf f}(A) \ar@{->}[r]^{\RHom_{A}(B,-)}  & \sfD^{\mathsf f}(B)} 
\]
The functor $\RHom_{A}(B,-)$ is an equivalence, because $\vf$ is a quasi-isomorphism. 
Finally, since $\nat B$ is finite dimensional over $k$, it is finite when viewed as a module over $\nat A$ via $\vf$, so the functor $\Hom_{A}(B,-)$ preserves coproducts. It remains to apply
Lemma~\ref{le:equivalence} to deduce that $\Hom_{A}(B,-)$ is an equivalence.
\end{proof}

In the remainder of this section we discuss  stratification for homotopy categories of graded-injective \dg modules.

\subsection*{$\KInj A$ is tensor triangulated}\ \medskip

\noindent
Given a \dg Hopf algebra $A$ and \dg $A$-modules $M$ and $N$, there is a \dg $A$-module structure on $M\otimes_{k}N$, obtained by restricting the natural action of $A\otimes_{k}A$ along the comultiplication $A\to A\otimes_{k}A$. This is the \emph{diagonal action} of $A$ on $M\otimes_{k}N$.

\begin{proposition}
\label{pr:dgh-tt}
Let $A$ be a \dg $k$-algebra satisfying Hypothesis~\emph{\ref{hyp:A}} and $\ik$ a semi-injective resolution of $k$ over $A$. The tensor product $\otimes_{k}$ with diagonal $A$-action endows $\KInj A$ with a structure of a tensor triangulated category with unit $\ik$.
\end{proposition}

\begin{proof}
Standard arguments show that for any \dg $A$-module $M$ the graded
$\nat A$-modules underlying $M\otimes_{k}A$ and $A\otimes_{k}M$ are
free; see \cite[\S3.1]{Benson:1991a}. Graded-injective \dg $A$-modules
are graded-free as $\nat A$-modules and direct sums of
graded-injectives are graded-injectives, since $\nat A$ is
noetherian. Therefore if $I$ and $J$ are graded-injective \dg
$A$-modules, then so is the \dg $A$-module $I\otimes_{k}J$. Hence one
does get a tensor product on $\KInj A$.

We claim that the morphism $k\to \ik$ induces an isomorphism 
\[
I\cong k\otimes_{k} I \xrightarrow{\sim} \ik \otimes_{k} I 
\]
in $\KInj A$, so that $\ik$ is the unit of the tensor product on
$\KInj A$. Indeed, it is an isomorphism when $I=\ik$ because the
morphism $\ik \to \ik\otimes_{k}\ik$ is a quasi-isomorphism and both
$\ik$ and $\ik\otimes_{k}\ik$ are semi-injective \dg modules, the
first by construction and the second by
Lemma~\ref{le:injectives-filtration}. Therefore the map above is an
isomorphism for any $I$ in $\Loc(\ik)$, which is all of $\KInj A$, by
Theorem~\ref{th:kinj-generator}. For an alternative argument, see
\cite[Proposition~5.3]{Benson/Krause:2008a}.

The other requirements of a tensor triangulated structure are readily verified. 
 \end{proof}

\begin{remark}
\label{rem:dgh-action}
Let $A$ be a \dg $k$-algebra satisfying Hypothesis~\ref{hyp:A}.

The $k$-algebra $\Ext^{*}_{A}(k,k)$ is graded commutative, as $A$ is a Hopf $k$-algebra. Identifying $\Ext^{*}_A(k,k)$ with $\End^{*}_{\sfK}(\ik)$ there is a \emph{canonical action} on $\KInj A$ given by
\[ 
\Ext^{*}_A(k,k) \xrightarrow{X \otimes_k -} \Hom^{*}_{\KInj A}(X,X) 
\]
for $X$ in $\KInj{A}$. If the $k$-algebra $\Ext^{*}_{A}(k,k)$ is
finitely generated, and hence noetherian, the theory of localisation
and support described in \S\ref{se:strat} applies to $\KInj A$. Finite
generation holds, for instance, when the differential on $A$ is zero,
by a result of Friedlander and Suslin~\cite{Friedlander/Suslin:1997a}.
\end{remark}

Given a \dg $k$-algebra $A$ satisfying Hypothesis~\ref{hyp:A}, the
structure of $\KInj A$ which is most relevant for us does not depend
on the choice of a comultiplication on $A$. This is made precise in
the following proposition which is an immediate consequence of
Corollary~\ref{cor:2actions} and Lemma~\ref{le:2actions}.

\begin{proposition}
\label{pr:2tensors}
Let $A$ be a \dg $k$-algebra satisfying Hypothesis~\emph{\ref{hyp:A}}.
The following structures of $\KInj A$ do not depend on the choice of a comultiplication on $A$:
\begin{enumerate}[\quad\rm(1)]
\item the functors $\Gamma_{\mcV}$, $L_{\mcV}$, and $\Gamma_{\fp}$;
\item  the maps  $\sigma$ and $\tau$ defined in \emph{\eqref{eq:classification}};
\item stratification of $\KInj A$ via  the canonical action of $\Ext^{*}_A(k,k)$. \qed
\end{enumerate}
\end{proposition}

\subsection*{Transfer of stratification}\ \medskip

\noindent 
The next results deal with transfer of stratification between homotopy
categories of graded injective \dg modules of \dg algebras.

\begin{proposition}
\label{pr:strat-quism}
Let $\vf\col A\to B$ be a quasi-isomorphism of \dg $k$-algebras where
$A,B$ satisfy Hypothesis~\emph{\ref{hyp:A}}.  Then $\KInj A$ is
stratified by the canonical action of $\Ext^{*}_A(k,k)$ if and only if
$\KInj B$ is stratified by the canonical action of $\Ext^{*}_B(k,k)$.
\end{proposition}

Observe: $\vf$ is not required to commute with the comultiplications on $A$ and $B$.

\begin{proof}
Proposition~\ref{pr:kinj-quism} yields that $\Hom_A(B,-)\col \KInj A\rightarrow\KInj B$
is an equivalence sending  a semi-injective resolution of $k$ over $A$ to a semi-injective
resolution of $k$ over $B$. Thus $\Hom_A(B,-)$ induces an isomorphism
$\mu\col\Ext^{*}_A(k,k)\xrightarrow{\sim} \Ext^{*}_B(k,k)$.  Observe
that $\KInj B$ admits two actions of $\Ext^{*}_A(k,k)$.  The first is
the canonical action of $\Ext^{*}_B(k,k)$ composed with $\mu$ and the
other is the canonical action on $\KInj A$ composed with the equivalence $\Hom_A(B,-)$.

Suppose $\KInj A$ is stratified by the canonical action of
$\Ext^{*}_A(k,k)$. Then $\KInj B$ is stratified by $\Ext^{*}_A(k,k)$ via the
second action because $\Hom_A(B,-)$ is an equivalence.  It follows
from Lemma~\ref{le:2actions} that $\KInj B$ is stratified via the
first action, and hence the canonical action of $\Ext^{*}_B(k,k)$
stratifies $\KInj B$ since $\mu$ is an isomorphism. This argument can
be reversed by using a quasi-inverse of $\Hom_A(B,-)$.
\end{proof}

\begin{theorem}
\label{th:strat-descent}
Let $A$ be a \dg $k$-algebra satisfying Hypothesis \emph{\ref{hyp:A}}. If $\KInj A$ is stratified by the canonical action of $\Ext^{*}_A(k,k)$, then $\KInj{A^{0}}$ is stratified by the canonical action of $\Ext^{*}_{A^{0}}(k,k)$.
\end{theorem}

\begin{proof}
We write $\sfK(A^{0})$ and $\sfK(A)$ for $\KInj{A^{0}}$ and $\KInj A$, respectively. Hypothesis~\ref{hyp:A} implies that $\nat A$, the graded module underlying $A$,  is free of finite rank over $A^{0}$. Therefore when $I$ is a graded-injective (respectively, semi-injective) \dg $A$-module the adjunction isomorphism $\Hom_{A^{0}}(-,I)\cong \Hom_{A}(A\otimes_{A^{0}}-,I)$ yields that $I$ is also graded-injective (respectively, semi-injective) as a \dg $A^{0}$-module.  In particular, the inclusion $A^{0} \to A$ gives rise to a restriction functor 
\[
(-){\da}\col \sfK(A)\to \sfK(A^{0}).
\]
Let $\ik$ be a semi-injective resolution of $k$ over $A$. The \dg $A^{0}$-module $\ik{\da}$ is  then  a semi-injective resolution of $k$ over $A^{0}$, so restriction induces a homomorphism of graded $k$-algebras $\Ext^{*}_{A}(k,k)\to \Ext^{*}_{A^{0}}(k,k)$. In view of Lemma~\ref{le:2rings} it thus suffices to prove that $\sfK(A^{0})$ is stratified by the action of $\Ext^{*}_{A}(k,k)$.

The restriction functor has a right adjoint
\[ 
\Hom_{A^{0}}(A,-) \col\sfK(A^{0})\to\sfK(A). 
\]

Fix a prime $\fp$ in $\Spec \Ext^{*}_{A}(k,k)$, and let $X$ and $Y$ be objects in $\Gamma_\fp\sfK(A^{0})$ with $\Hom^*_{\sfK(A^{0})}(X,Y)=0$. It suffices to prove that $X$ or $Y$ is zero; see  Lemma~\ref{le:loc-minimal}.

As an $A^{0}$-module $\nat A$ is free of finite rank therefore the \dg
$A^{0}$-module $A\da$ is in $\Thick(A^{0})$, the thick subcategory of
$\sfK(A^{0})$ generated by $A^{0}$. Thus the \dg $A^{0}$-module
$\Hom_{A^{0}}(A,X)\da$ is in $\Thick(X)$. This justifies the second
isomorphism below:
\[ 
\Hom^*_{\sfK(A)}(\Hom_{A^{0}}(A,X),\Hom_{A^{0}}(A,Y))\cong
\Hom^*_{\sfK(A^{0})}(\Hom_{A^{0}}(A,X)\da,Y)=0. 
\]
The first one is adjunction. Observe that the adjunction isomorphism 
\[
\Hom^{*}_{\sfK(A^{0})}(\ik\da,X) \cong \Hom^{*}_{\sfK(A)}(\ik,\Hom_{A^{0}}(A,X))
\]
is compatible with the action of $\Ext^{*}_{A}(k,k)$. It thus follows that both $\Hom_{A^{0}}(A,X)$ and $\Hom_{A^{0}}(A,Y)$ are in $\Gamma_\fp\sfK(A)$, and hence that one of them is zero, since $\sfK(A)$ is stratified by $\Ext^{*}_{A}(k,k)$. Assume, without loss of generality, that $\Hom_{A^{0}}(A,X)=0$. The isomorphism above then yields $\Hom^{*}_{\sfK(A^{0})}(\ik\da,X)=0$. By Theorem~\ref{th:kinj-generator} the \dg $A^{0}$-module $\ik\da$ generates $\sfK(A^{0})$, hence one gets that $X=0$.
\end{proof}

\section{Graded polynomial algebras}
\label{se:poly}
Let $k$ be a field and $S$ a graded polynomial $k$-algebra on finitely
many indeterminates. We assume that each variable is of even degree if
the characteristic of $k$ is not $2$, so that $S$ is strictly
commutative. The algebra $S$ is viewed as a \dg algebra with zero
differential and we write $\sfD(S)$ for the derived category of \dg
$S$-modules. The objects of this category are \dg $S$-modules and the
morphisms are obtained by inverting the quasi-isomorphisms; see for
example \cite{Keller:1994a}.
 
The main result of this section is a classification of the localising
subcategories of $\sfD(S)$. This is a graded analogue of the theorem
of Neeman \cite{Neeman:1992a}. In \cite{Benson/Iyengar/Krause:bik2},
we establish this result for general graded commutative noetherian
rings. The argument presented here for $S$ is simpler because we make
use of the fact that the Koszul complex of a regular local ring is
quasi-isomorphic to its residue field.

The category $\sfD(S)$ is tensor triangulated where the tensor product
is $\lotimes_{S}$, the derived tensor product, and the unit is
$S$. Observe that $S$ is compact and that it generates $\sfD(S)$. In
particular, localising subcategories of $\sfD(S)$ are also tensor
ideal. There is a canonical action of the ring $S$ on $\sfD(S)$, where
the homomorphism $S\to \End^{*}_{\sfD(S)}(X)$ is given by
multiplication. The theory of localisation and support described in
\S\ref{se:strat} thus applies.

In this context, one has also the following useful identification.

\begin{lemma}
\label{le:dgloc}
Let $\fp$ be a point in $\Spec S$ and set $\mcZ=\{\fq\in\Spec S\mid
\fq\not\subseteq\fp\}$. For each $M$ in $\sfD(S)$ there is a natural
isomorphism $L_{\mcZ}M\cong M_{\fp}$.
\end{lemma}

\begin{proof}
The functor on $\sfD(S)$ defined by $M\mapsto M_{\fp}$ is a
localisation functor and has the same acyclic objects as $L_{\mcZ}$,
since $H^{*}(L_{\mcZ}M)\cong H^{*}(M)_{\fp}$, by \cite[Theorem
4.7]{Benson/Iyengar/Krause:2008a}. This implies the desired result.
\end{proof}

\begin{theorem}
\label{th:DdgS}
The category $\sfD(S)$ is stratified by the canonical $S$-action. In particular, the maps
\[ 
\xymatrix{
\{\text{localising subcategories of $\sfD(S)$}\}
\ar@<.5ex>[r]^-\sigma &
\{\text{subsets of $\Spec S$}\} \ar@<.5ex>[l]^-\tau} 
\]
described in \eqref{eq:classification} are mutually inverse bijections.
\end{theorem}

\begin{proof}
By Theorem~\ref{th:transfer-strat}, the second part of the statement follows from the first since localising subcategories of $\sfD(S)$  are tensor ideal.

Fix a point $\fp\in\Spec S$. Since $\Hom_{\sfD(S)}^{*}(S,X)$ equals
$H^{*}(X)$, the subcategory $\Gamma_{\fp}\sfD(S)$ consists of \dg
modules whose cohomology is $\fp$-local and $\fp$-torsion.  Let
$k(\fp)$ be the homogeneous localisation of $S/\fp$ at $\fp$; it is a
graded field. Evidently $k(\fp)$ is in $\Gamma_{\fp}\sfD(S)$, so for
the desired result it suffices to prove that there is an equality
\[
\Loc(M) = \Loc (k(\fp)).
\] 
for any non-zero \dg module $M$ in $\Gamma_{\fp}\sfD(S)$. We verify this first for $M=\Gamma_{\fp}S$. 

Let $\bss=s_{1},\dots,s_{d}$ be a sequence of elements in $S$ whose
images in $S_{\fp}$ are a minimal set of generators for the ideal $\fp
S_{\fp}$ in the ring $S_{\fp}$. Let $\mcV$ denote the set of primes in
$\Spec S$ containing $\bss$ and let $\mcZ=\{\fq\in\Spec S\mid
\fq\not\subseteq\fp\}$; these are specialisation closed
subsets. Observe that $\mcV\setminus\mcZ=\{\fp\}$, so one gets the
first equality below:
\[
\Loc(\Gamma_{\fp}S) 
= \Loc(L_{\mcZ}\Gamma_{\mcV}S) 
=  \Loc(L_{\mcZ}(\kos S{\bss})) 
= \Loc((\kos S{\bss})_{\fp})
=\Loc(k(\fp)).
\]
The second equality is a consequence of
Proposition~\ref{pr:loc-pr}(2), since the functor $L_{\mcZ}$ preserves
arbitrary coproducts, while the third one follows from
Lemma~\ref{le:dgloc}. The last equality holds because the \dg module
$(\kos S{\bss})_{\fp}$ is quasi-isomorphic to $k(\fp)$ by the choice
of $\bss$, since it is a regular sequence in $S_{\fp}$; see
\cite[Corollary~1.6.14]{Bruns/Herzog:1998a}.

We now know that $\Loc(\Gamma_{\fp}S)=\Loc(k(\fp))$ holds. Applying
the functor $-\lotimes_{S} M$ to it yields the second equality below:
\[
\Loc(M) = \Loc(\Gamma_{\fp}M) =\Loc(k(\fp)\lotimes_{S} M).
\]
The first one holds because $M$ is in $\Gamma_{\fp}\sfD(S)$. The action of $S$ on $k(\fp)\lotimes_{S} M$ factors through the graded field $k(\fp)$, as $S$ is commutative. The equality above implies that $H^{*}(k(\fp)\lotimes_{S} M)$ is non-zero, so one deduces that $k(\fp)\lotimes_{S} M$ is quasi-isomorphic to a direct sum of shifts of $k(\fp)$. Hence $\Loc(k(\fp)\lotimes_{S} M)=\Loc(k(\fp))$. Combining this equality with the one above yields the desired result.
 \end{proof}

\section{Exterior algebras}
\label{se:bgg}

Let $k$ be a field and let $\Lambda$ be the graded exterior algebra
over $k$ on indeterminates $\xi_{1},\dots,\xi_{c}$ of negative odd
degree. We view $\Lambda$ as a \dg algebra with zero differential. The
main result of this section is a classification of the localising
subcategories of the homotopy category of graded-injective \dg
$\Lambda$-modules. It will be deduced from Theorem~\ref{th:DdgS}, via
a \dg Bernstein-Gelfand-Gelfand correspondence from
\cite[\S7]{Avramov/Buchweitz/Iyengar/Miller:hffc}.

\begin{defn}
\label{defn:traceform}
Let $S$ be a graded polynomial algebra over $k$ on indeterminates $x_{1},\dots,x_{c}$ with $|x_{i}|=-|\xi_{i}|+1$ for each $i$. The $k$-algebra $\Lambda\otimes_{k} S$ is graded commutative; view it as a \dg algebra with zero differential. In it consider the element
\[ 
\delta = \sum_{i=1}^{c}\xi_{i}\otimes_{k} x_{i} 
\]
of degree $1$. It is easy to verify that $\delta^{2}=0$ holds. In what follows $J$ denotes the \dg module over $\Lambda\otimes_{k}S$ with underlying graded module and differential given by
\[
\nat J = \Hom_{k}(\Lambda,k)\otimes_{k} S\quad\text{and}\quad d(e)=\delta e.
\]
Observe that since $J$ is a \dg module over $\Lambda\otimes_{k}S$, for each \dg module $M$ over 
$\Lambda$ there is an induced structure of a \dg $S$-module on $\Hom_{\Lambda}(J,M)$. 
\end{defn}

The result below builds on \cite[Theorem
7.4]{Avramov/Buchweitz/Iyengar/Miller:hffc}; see Remark~\ref{rem:bgg}
below. Recall that $\KInj \Lambda$ is the homotopy category of
graded-injective \dg modules over $\Lambda$ and $\sfD(S)$ is the
derived category of \dg modules over $S$.

\begin{theorem}
\label{th:bgg}
The \dg $(\Lambda\otimes_{k}S)$-module $J$ in Definition \emph{\ref{defn:traceform}} has these properties:
\begin{enumerate}[\quad\rm(1)]
\item There is a quasi-isomorphism $k\to J$ of \dg $\Lambda$-modules and $J$ is semi-injective.
\item The map $S\to \Hom_{\Lambda}(J,J)$ induced by right multiplication is a morphism of \dg $k$-algebras and a quasi-isomorphism.
\end{enumerate}
For any \dg $(\Lambda\otimes_{k}S)$-module $J$ satisfying these conditions the functor
\[
\Hom_{\Lambda}(J,-)\col \KInj{\Lambda} \to \sfD(S) 
\]
is an equivalence of triangulated categories.
\end{theorem}

\begin{proof}
The surjection $\Lambda\to k$ is a morphism of \dg $\Lambda$-modules
and hence so is its dual $k\to \Hom_{k}(\Lambda,k)$. Combined with the
map of $k$-vector spaces $k\to S$ one gets a morphism $k\to
\Hom_{k}(\Lambda,k)\otimes_{k}S=J$ of \dg $\Lambda$-modules. The
module $J$ is precisely the \dg module $X$ from
\cite[\S7.3]{Avramov/Buchweitz/Iyengar/Miller:hffc}. It thus follows
from \cite[\S\S7.6.2,7.6.5]{Avramov/Buchweitz/Iyengar/Miller:hffc}
that $k\to J$ is a quasi-isomorphism and that $J$ is semi-injective
as a \dg $\Lambda$-module. Moreover, the map $S\to
\Hom_{\Lambda}(J,J)$ is a quasi-isomorphism by \cite[Theorem
7.4]{Avramov/Buchweitz/Iyengar/Miller:hffc}. The module $J$ thus has
the stated properties.

Let $J$ be any \dg $(\Lambda\otimes_{k}S)$-module satisfying conditions (1) and (2) in the statement of the theorem. It is easy to verify that the functor $\Hom_{\Lambda}(J,-)$ from $\KInj{\Lambda}$ to $\sfD(S)$ is exact. We claim that $J$ is compact and generates the triangulated category $\KInj{\Lambda}$. One way to prove this is to note that $\Lambda$ satisfies Hypothesis~\ref{hyp:A}, with comultiplication defined by $\xi_{i}\mapsto \xi_{i}\otimes 1 + 1 \otimes \xi_{i}$, so that Theorem~\ref{th:kinj-generator} applies. Compactness yields that the functor $\Hom_{\Lambda}(J,-)$ preserves coproducts, since a quasi-isomorphism between \dg
$S$-modules is an isomorphism in $\sfD(S)$. Furthermore, as $S$ is a compact generator for $\sfD(S)$, condition (2) provides the hypotheses required to apply Lemma~\ref{le:equivalence}, which yields that $\Hom_\Lambda(J,-)$ is an equivalence.
\end{proof}

\begin{remark}
\label{rem:bgg}
Let $F \col \sfD(S)\to \sfK(S)$ be a left adjoint to the canonical localisation functor $\sfK(S)\to \sfD(S)$, and $-\lotimes_{S}J$ the composite functor 
\[ 
\sfD(S)\xrightarrow{F}\sfK(S) \xrightarrow{-\otimes_{S}J}\KInj \Lambda.
\]
The proof of Theorem~\ref{th:bgg} shows that $-\lotimes_{S}J$ is left
adjoint to $\Hom_{\Lambda}(J,-)$, so restricting the equivalence in
Theorem~\ref{th:bgg} to compact objects yields the equivalence
$\sfD^{\mathsf f}(\Lambda)\xra{\sim}\sfD^{\mathsf f}(S)$ contained in
\cite[Theorem 7.4]{Avramov/Buchweitz/Iyengar/Miller:hffc}.
\end{remark}

As in the proof of Theorem~\ref{th:bgg} consider $\Lambda$ as \dg Hopf $k$-algebra with
\[ 
\Delta(\xi_i)=\xi_i\otimes 1 + 1 \otimes \xi_i. 
\]
Observe that $\Lambda$ satisfies Hypothesis~\ref{hyp:A}, so the triangulated category $\KInj \Lambda$ has a canonical tensor triangulated structure, by Proposition~\ref{pr:dgh-tt}, and hence a canonical action of $\Ext^{*}_{\Lambda}(k,k)$; see Remark~\ref{rem:dgh-action}. Moreover, the $k$-algebra $\Ext^{*}_{\Lambda}(k,k)$ is noetherian, by Theorem~\ref{th:bgg}(2).

\begin{theorem}
\label{th:Lambda-strat}
The tensor triangulated category $\KInj{\Lambda}$ is stratified by the canonical action of $\Ext^{*}_{\Lambda}(k,k)$. Therefore the maps
\[ 
\left\{
\begin{gathered}
\text{localising subcategories of $\KInj{\Lambda}$}
\end{gathered}
\right\}
\xymatrix{
\ar@<1ex>[r]^\sigma & \ar@<1ex>[l]^\tau}
\{\text{subsets of $\Spec \Ext^{*}_{\Lambda}(k,k)$}\} 
\]
described in \eqref{eq:classification} are mutually inverse bijections.
\end{theorem}

\begin{proof}
Let $S$ and $J$ be as in Definition~\ref{defn:traceform}. In view of
Theorem~\ref{th:bgg}(1) we identify $\Ext^{*}_{\Lambda}(k,k)$ and
$\Hom_{\KInj \Lambda}(J,J)$. Theorem~\ref{th:DdgS} applies to $S$ and
yields that $\sfD(S)$ is stratified by the canonical $S$-action on it,
and hence also by an action of $\Ext^{*}_{\Lambda}(k,k)$ obtained from
the isomorphism $S\cong \Ext^{*}_{\Lambda}(k,k)$ in
Theorem~\ref{th:bgg}(2).  Theorem~\ref{th:bgg} provides an equivalence
of triangulated categories $\KInj{\Lambda} \to \sfD(S)$ that sends
$J$, the compact generator of $\KInj \Lambda$, to $S$, the compact generator
of $\sfD(S)$.  Hence $\KInj{\Lambda}$ is stratified by the canonical action of
$\Ext^{*}_{\Lambda}(k,k)$, by Lemma~\ref{le:2actions}.

By Theorem~\ref{th:transfer-strat} the stated bijection is a consequence of the stratification.
\end{proof}

\section{A Koszul \dg algebra}
\label{se:Lambda-to-A}

In this section we classify the localising subcategories of the homotopy category of graded-injective \dg modules over the Koszul \dg algebra of the group algebra of an elementary abelian group. To this end we establish an equivalence with the corresponding homotopy category over an exterior algebra, covered by Theorem~\ref{th:Lambda-strat}.

Let $E$ be an elementary abelian $p$-group of rank $r$ and let $k$ be a field of characteristic $p$. The group algebra $kE$ is thus of the form
\[ 
kE=k[z_1,\dots,z_r]/(z_1^p,\dots,z_r^p). 
\]
Let $A$ be the \dg algebra with $\nat A$ an exterior algebra over $kE$ on generators $y_{1},\dots,y_{r}$ each of degree $-1$, and differential defined by $d(y_{i})=z_{i}$, so that $d(z_{i})=0$. This is the Koszul complex of $A$, viewed as a \dg algebra; see \cite[\S1.6]{Bruns/Herzog:1998a}.

The group algebra $kE$ is an example of a complete intersection, and
the Koszul \dg algebra of such a ring is formal, with cohomology an
exterior algebra; see \cite[\S2.3]{Bruns/Herzog:1998a}. This
computation is straightforward for the case of $kE$ and is given
below.

\begin{lemma}
\label{le:qi}
Let $\Lambda$ be an exterior algebra over $k$ on indeterminates $\xi_{1},\dots,\xi_{r}$ of degree $-1$, viewed as a \dg algebra with zero differential. The morphism $\vf\col \Lambda\to A$ of \dg $k$-algebras defined by $\vf(\xi_{i})=z_i^{p-1}y_i$ is a quasi-isomorphism. In particular, the $k$-algebra $\Ext^{*}_{A}(k,k)$ is a polynomial ring in $r$ indeterminates of degree $2$.
\end{lemma}

\begin{proof}
A routine calculation shows that $\vf$ is a morphism of \dg
$k$-algebras. What needs to be verified is that it is a
quasi-isomorphism, and this is determined only by the structure of
$\Lambda$ and $A$ as complexes of $k$-vector spaces.

Let $\Lambda(i)$ be the exterior algebra on the variable $\xi_{i}$ and $A(i)$ the Koszul \dg algebra over $k[z_{i}]/(z_{i}^{p})$, with exterior generator $y_{i}$. Observe that $\vf=\vf(1)\otimes_{k}\cdots\otimes_{k}\vf(r)$ where $\vf(i)\col \Lambda(i)\to A(i)$ is the morphism of complexes mapping $\xi_{i}$ to $z_{i}^{p-1}y_{i}$. Each $\vf(i)$ is a quasi-isomorphism, by inspection, and hence so is $\vf$.

Since $\vf$ is a quasi-isomorphism the $k$-algebras $\Ext^{*}_{\Lambda}(k,k)$ and $\Ext^{*}_{A}(k,k)$ are isomorphic.  Theorem~\ref{th:bgg}(2) implies $\Ext^{*}_{A}(k,k)$ has the stated structure.
\end{proof}

We endow the  \dg algebra $A$ with a comultiplication
\[ 
\Delta(z_i)=z_i\otimes 1 + 1 \otimes z_i\quad\text{and}\quad
\Delta(y_i)=y_i\otimes 1 + 1 \otimes y_i.
\]
With this structure $A$ satisfies Hypothesis~\ref{hyp:A}. The homotopy
category of graded-injective \dg $A$-modules $\KInj A$ has thus a
tensor triangulated structure and a canonical action of
$\Ext^{*}_{A}(k,k)$; see Proposition~\ref{pr:dgh-tt} and
Remark~\ref{rem:dgh-action}.

\begin{theorem}
\label{th:strat-A}
The tensor triangulated category $\KInj{A}$ is stratified by the canonical action of  $\Ext^{*}_{A}(k,k)$.
Therefore the maps
\[ 
\left\{
\begin{gathered}
\text{localising subcategories of $\KInj{A}$}
\end{gathered}
\right\}
\xymatrix{
\ar@<1ex>[r]^\sigma & \ar@<1ex>[l]^\tau}
\{\text{subsets of $\Spec \Ext^{*}_{A}(k,k)$}\} 
\]
described in \eqref{eq:classification} are mutually inverse bijections.
\end{theorem}

\begin{proof}
Let $\Lambda$ be the exterior algebra from Lemma~\ref{le:qi}. With
comultiplication defined by $\xi_{i}\mapsto \xi_{i}\otimes 1 +
1\otimes \xi_{i}$ this \dg algebra satisfies Hypothesis~\ref{hyp:A}.
The desired result thus follows from Lemma~\ref{le:qi},
Theorem~\ref{th:Lambda-strat} and Proposition~\ref{pr:strat-quism}.
\end{proof}

\section{Elementary abelian groups}
\label{se:A-to-kE}

We classify the localising subcategories of the
homotopy category of injective modules over the group algebra of an
elementary abelian group, by deducing it from the corresponding
statement for its Koszul \dg algebra in Theorem~\ref{th:strat-A}.

Let $E$ be an elementary abelian $p$-group of rank $r$ and $k$ a field
of characteristic $p$.  In this section we view its group algebra
$kE$, which is isomorphic to
\[
k[z_{1},\dots,z_{r}]/(z_{1}^{p},\dots,z_{r}^{p})
\]
as a \dg algebra over $k$ with zero differential. The diagonal map of
$E$ endows $kE$ with a structure of a \dg Hopf $k$-algebra with
comultiplication
\[ 
\Delta(z_i) = z_i \otimes 1 +  z_i \otimes z_i + 1 \otimes z_i.
\]
We call this the \emph{group Hopf structure} on $kE$; it is evidently cocommutative. This Hopf structure is needed in \S\ref{se:kE-to-kG} for the passage from $E$ to a general finite group.

We also have to consider a different cocommutative \dg Hopf algebra
structure on $kE$, where the comultiplication is defined by
\[ 
\Delta(z_i) = z_i\otimes 1 + 1 \otimes z_i. 
\]
We call this the \emph{Lie Hopf structure} on $kE$. It comes from regarding $kE$ as the restricted universal enveloping algebra of an abelian $p$-restricted Lie algebra with zero $p$-power map. The Lie Hopf structure has the advantage that the inclusion of $kE$ into its Koszul \dg algebra is a map of \dg Hopf algebras; this is exploited in the proof of Theorem~\ref{th:strat-kE} below.

The group Hopf structure and the Lie Hopf structure on $kE$ both
satisfy Hypothesis~\ref{hyp:A}, and thus give rise to two tensor
triangulated structures on $\KInj {kE}$. Thus there are two actions of
$\Ext^{*}_{kE}(k,k)$ on $\KInj {kE}$, which we call the group action and
the Lie action, respectively; see Proposition~\ref{pr:dgh-tt} and
Remark~\ref{rem:dgh-action}. The $k$-algebra
$\Ext^{*}_{kE}(k,k)$ is finitely generated and hence noetherian.  Thus the theory of localisation
and support described in \S\ref{se:strat} applies.

\begin{theorem}
\label{th:strat-kE}
The triangulated category $\KInj{kE}$ is stratified by both the group action and the Lie action of $\Ext^{*}_{kE}(k,k)$. The maps $\sigma$ and $\tau$ described in \eqref{eq:classification} do not depend on the action used, and give mutually inverse bijections: 
\[ 
\xymatrix{\{\text{localising subcategories of $\KInj{kE}$}\}
\ar@<.5ex>[r]^-\sigma &
\{\text{subsets of $\Spec \Ext^{*}_{kE}(k,k)$}\}.\ar@<.5ex>[l]^-\tau} 
\]
\end{theorem}

\begin{proof}
Proposition~\ref{pr:2tensors} justifies the statement about independence of actions.  Hence it suffices to consider the Lie Hopf structure on $kE$.

Let $A$ be the Koszul \dg algebra of $kE$ with structure of \dg Hopf $k$-algebra introduced in \S\ref{se:Lambda-to-A}. Observe that $kE=A^{0}$ and that the inclusion $kE \to A$ is compatible with the Lie Hopf structure; this is the reason for working with this Hopf structure on $kE$. It now remains to apply Theorems~\ref{th:strat-A} and \ref{th:strat-descent}.
\end{proof}

\section{Finite groups}
\label{se:kE-to-kG}

In this section we prove that the tensor ideal localising subcategories of the homotopy category of complexes of injective modules over the group algebra of a finite group are stratified by the cohomology of the group.
This is achieved by descending to the case of an elementary abelian group, covered by Theorem~\ref{th:strat-kE}. The crucial new input required here is Quillen's stratification theorem~\cite{Quillen:1971b,Quillen:1971c}.

Let $G$ be a finite group and $k$ a field of characteristic dividing
the order of $G$. The group algebra $kG$ is a Hopf algebra where the
comultiplication is defined by $\Delta(g)=g\otimes g$ for each $g\in
G$.  We consider the homotopy category of complexes of injective
$kG$-modules $\KInj{kG}$.  The diagonal action of $G$ induces a tensor
triangulated structure with unit the injective resolution $ik$ of $k$;
see \cite[\S5]{Benson/Krause:2008a} and also Proposition
\ref{pr:dgh-tt}. As is customary, $H^*(G,k)$ denotes the cohomology of
$G$, which is the $k$-algebra $\Ext^{*}_{kG}(k,k)$. This algebra is
finitely generated, and hence noetherian, by a result of Evens and Venkov \cite{Evens:1961a,Venkov:1959a}; see also Golod \cite{Golod:1959a}. It acts on
$\KInj{kG}$ via the canonical action, and the theory described in
\S\ref{se:strat} applies. We write $\mcV_G$ for
$\Spec H^*(G,k)$ and $\mcV_G(X)$ for the support of any complex $X$ in
$\KInj{kG}$.

For each subgroup $H$ of $G$ restriction yields a homomorphism of graded rings $\res_{G,H}\col H^*(G,k) \to H^*(H,k)$, and hence a map on $\Spec$:
\[ 
\res_{G,H}^*\col \mcV_H \to \mcV_G.
\] 

Part of Quillen's theorem\footnote{See the discussion following
Proposition 11.2 of \cite{Quillen:1971c}.} is that for each $\fp$ in
$\mcV_G$ there exists an elementary abelian subgroup $E$ of $G$ such
that $\fp$ is in the image of $\res_{G,E}^*$. We say that $\fp$
\emph{originates} in such an $E$ if there does not exist a proper
subgroup $E'$ of $E$ such that $\fp$ is in the image of
$\res_{G,E'}^*$. In this language, \cite[Theorem~10.2]{Quillen:1971c}
reads:

\begin{theorem}
\label{th:quillen}
For each $\fp\in\mcV_G$, the pairs $(E,\fq)$ where $\fp=\res_{G,E}^*(\fq)$ and such that $\fp$ originates in $E$ are all $G$-conjugate. This sets up a one to one correspondence between primes $\fp$ in $\mcV_G$ and 
$G$-conjugacy classes of such pairs $(E,\fq)$. \qed
\end{theorem}

To make use of this we need a subgroup theorem for elementary abelian
groups. As usual given a subgroup $H\le G$ there are exact functors
\[
(-)\da_{H}\col \KInj{kG} \to \KInj{kH} \quad\text{and} \quad
(-)\ua^{G}\col \KInj{kH} \to \KInj{kG}
\]
defined by restriction and induction, $-\otimes_{kH}kG$, respectively. The functor $(-)\ua^{G}$ is faithful and left adjoint to $(-)\da_{H}$. For each object $X$ in $\KInj{kG}$ and object $Y$ in $\KInj{kH}$ there is a natural isomorphism
\begin{equation}
\label{eq:indres}
(X\da_{H}\otimes_{k}Y)\ua^{G}\cong X\otimes_{k} Y\ua^{G}
\end{equation}
in $\KInj{kG}$, where an element $(x\otimes y)\otimes g$ is mapped to $xg\otimes (y\otimes g)$.

\begin{lemma}
\label{le:kappa-res}
Let $H$ be a subgroup of $ G$. Fix $\fp\in\mcV_G$ and set $\mathcal U=(\res_{G,H}^*)^{-1}\{\fp\}$.  
\begin{enumerate}[\quad\rm(1)]
\item 
For any $X\in\KInj{kG}$ there is an isomorphism $(\Gamma_\fp X)\da_H\cong\bigoplus_{\fq\in\mathcal U} \Gamma_\fq (X\da_H)$.
\item
For any $Y\in\KInj{kH}$ there is an isomorphism $\Gamma_\fp(Y\ua^G)\cong \bigoplus_{\fq\in\mathcal U} (\Gamma_\fq Y)\ua^G$.
\end{enumerate}
\end{lemma}

\begin{proof}
(1) For $X=ik$, this follows as in \cite[Lemma 8.2 and Proposition 8.4
  (iv)]{Benson/Carlson/Rickard:1996a}.  The general case is deduced as
follows:
\begin{align*} 
(\Gamma_\fp X)\da_H &\cong (\Gamma_\fp ik\otimes_k
X)\da_H\cong(\Gamma_\fp ik)\da_H\otimes_k X\da_H\\
&\cong\bigoplus_{\fq\in\mathcal U} \Gamma_\fq ik\otimes_k
X\da_H\cong\bigoplus_{\fq\in\mathcal U} \Gamma_\fq (X\da_H).
\end{align*}

(2) Using part (1) and \eqref{eq:indres}, one gets isomorphisms
\begin{align*} 
\Gamma_\fp (Y\ua^G)&\cong\Gamma_\fp ik \otimes_k Y\ua^G 
\cong((\Gamma_\fp ik)\da_H \otimes_k Y)\ua^G \\
&\cong\bigoplus_{\fq\in\mathcal U}(\Gamma_\fq ik \otimes_k Y)\ua^G 
\cong\bigoplus_{\fq\in\mathcal U}(\Gamma_\fq Y)\ua^G.
\qedhere
\end{align*}
\end{proof}

\begin{proposition}
\label{pr:V-ind}
Let $H$ be a subgroup of $G$. For any object $X$ in $\KInj{kG}$ and object $Y$ in $\KInj{kH}$ one has
\[
\mcV_G(X\da_H\ua^G)\subseteq\mcV_G(X)\quad\text{and}\quad
\mcV_G(Y\ua^G)=\res_{G,H}^*\mcV_H(Y).
\]
\end{proposition}

\begin{proof}
Let $W$ be an injective resolution of the permutation module, $k(G/H)$, on the cosets of $H$ in $G$. We claim that the natural map $k(G/H)\to W$ of complexes of $kG$-modules induces an isomorphism $X\da_H\ua^G \cong X\otimes_k W$ in $\KInj{kG}$.

Indeed, this is easy to check when $X=iM$, the injective resolution of a finite dimensional $kG$-module $M$, for then both $iM\da_{H}\ua^{G}$ and $iM\otimes_{k}W$ are injective resolutions of $M\otimes_{k}k(G/H)$. It then remains to note that complexes of the form $iM$ are a set of compactly generators for $\KInj{kG$}; see \cite[Lemmas 2.1 and 2.2]{Krause:2005a}.

When $\fp\in\mcV_G(X \otimes_k W)$ holds, since $\Gamma_\fp(X \otimes_k W)\cong \Gamma_\fp X \otimes_k W$ one gets $\Gamma_\fp X \ne 0$, that is to say, $\fp \in \mcV_G(X)$, as desired.

By Lemma~\ref{le:kappa-res}(2) the condition $\fp \in \mcV_G(Y\ua^G)$  is equivalent to: there exists $\fq\in\mcV_H$ such that $\res_{G,H}^*(\fq)=\fp$ and $\Gamma_\fq Y \ne 0$. Hence $\mcV_G(Y\ua^G)=\res_{G,H}^*\mcV_H(Y)$.
\end{proof}

The result below is an analogue of the subgroup theorem for an elementary abelian group $E$. Its proof is based on  the classification of localising subcategories of $\KInj{kE}$. The full version of the subgroup theorem, Theorem~\ref{th:subgp-full}, will be a consequence of the classification theorem for $\KInj{kG}$.

\begin{theorem}
\label{th:subgp-elem} 
Let $E'\le E$ be elementary abelian $p$-groups. For any object $X$ in $\KInj{kE}$ there is an equality
\[ 
\mcV_{E'}(X\da_{E'})=(\res_{E,E'}^*)^{-1}\mcV_E(X). 
\]
\end{theorem}

\begin{proof}
Fix a prime $\fq$ in $\mcV_{E'}$ and set
$\fp=\res_{E,E'}^*(\fq)$. Proposition~\ref{pr:V-ind} yields an
equality $\mcV_E(\Gamma_\fq ik\ua^E)=\{\fp\}=\mcV_E(\Gamma_\fp
ik)$. It thus follows from the classification of localising
subcategories for $kE$ in Theorem~\ref{th:strat-kE} that there is an
equality
\[
\Loc(\Gamma_\fq ik\ua^E)=\Loc(\Gamma_\fp ik).
\]
This implies that $\Gamma_\fq ik\ua^E \otimes_k X\ne 0$ holds if and only if $\Gamma_\fp ik \otimes_k X \ne 0$. The desired result follows from the chain of implications:
\begin{align*}
\Gamma_\fq(X\da_{E'})\ne 0 &\iff \Gamma_\fq ik \otimes_k X\da_{E'} \ne 0  \iff \\ 
 &\Gamma_\fq ik\ua^E \otimes_k X \ne 0 
\iff \Gamma_\fp ik \otimes_k X \ne 0  \iff \Gamma_\fp X \ne 0. 
\end{align*}
The second implication follows from \eqref{eq:indres} and the fact
that $(-)\ua^{E}$ is faithful.
\end{proof}

Next we formulate a version of Chouinard's theorem for $\KInj{kG}$. Recall
that for any ring $A$, a complex $P$ of projective $A$-modules is
\emph{semi-projective} if the functor $\Hom_\sfK(P,-)$, where $\sfK$
is the homotopy category of complexes of $A$-modules, takes
quasi-isomorphisms to isomorphisms.

\begin{proposition}
\label{pr:Chou}
Let $G$ be a finite group and $\fm$ the ideal of positive degree elements of $H^*(G,k)$.
The following statements hold for each  $X$ in $\KInj{kG}$.
\begin{enumerate}[\quad\rm(1)]
\item $\mcV_G(X)\subseteq\Spec H^*(G,k)\setminus \{\fm\}$ if and only if $X$ is acyclic.
\item $\mcV_G(X)\subseteq\{\fm\}$ if and only if $X$ is semi-projective. 
\item $X= 0$ if and only if $X\da_E= 0$ for every elementary abelian subgroup $E\le G$.
\end{enumerate}
\end{proposition}

\begin{proof}
We write $pk$ for a projective resolution and $tk$ for a Tate resolution of the $kG$-module $k$. These fit into an exact triangle $pk\to ik\to tk\to$ which is the localisation triangle \eqref{eq:localisation} in $\KInj{kG}$ for $ik$ with respect to the closed subset $\{\fm\}$. In particular $\Gamma_\fm ik=pk$, so that
$\Gamma_\fm X\cong pk\otimes_k X$.

\begin{Claim}
The induced map $pk\otimes_{k}X\to X$ is a semi-projective resolution of $X$.
\end{Claim}
Indeed, since $pk$ is a complex of projectives, so is $pk\otimes_{k}X$; semi-projectivity follows from the isomorphism $\Hom_{\sfK}(pk\otimes_{k}X,-)\cong \Hom_{\sfK}(pk,\Hom_{k}(X,-))$.

Statements (1) and (2) are immediate from the preceding claim.

(3) Suppose $X$ is non-zero in $\KInj{kG}$.  In view of the localisation triangle $pk\otimes_k X\to X\to tk\otimes_k X\to$ we may assume that $pk\otimes_k X\neq 0$ or $tk\otimes_k X\neq 0$. 

Assume that $pk\otimes_k X$ is non-zero; equivalently, that it is not acyclic, by the preceding claim. Observe that for each subgroup $H\le G$ the restriction functor $(-)\da_H$ sends semi-projectives to semi-projectives. Therefore $(pk\otimes_k X)\da_H$, and hence $X\da_{H}$, is non-zero.  

Now assume that $tk\otimes_k X$ is non-zero. Chouinard's theorem \cite{Chouinard:1976a} applies, because one can identify the subcategory of acyclic complexes in $\KInj{kG}$ with $\StMod(kG)$, and this identification is compatible with restrictions. This yields an elementary abelian subgroup $E\leq G$ with $(tk\otimes_k X)\da_E$, and hence also $X\da_E$, non-zero.
\end{proof}

The following result is a culmination of the development in
\S\S4--9. Its applications are deferred to ensuing sections.

\begin{theorem}
\label{th:KInj-min}
Let $G$ be a finite group. The triangulated category $\KInj{kG}$ is stratified by the canonical action of the cohomology algebra $H^{*}(G,k)$.
\end{theorem}

\begin{proof}
For any subgroup $H\le G$ we abbreviate $\KInj{kH}$ to $\sfK(kH)$. We
have to prove that for each $\fp\in\mcV_G$, the subcategory
$\Gamma_\fp\sfK(kG)$ is minimal among tensor ideal localising
subcategories of $\sfK(kG)$.

Let $X$ be a non-zero object in $\Gamma_\fp\sfK(kG)$. Proposition~\ref{pr:Chou} provides an elementary abelian subgroup $E_0$ of $G$ such that $X\da_{E_0}$ is non-zero. Choose a prime $\fq_0$ in $\mcV_{E_0}(X\da_{E_0})$. Using Proposition~\ref{pr:V-ind} one thus obtains
\[
\res_{G,E_0}^*(\fq_0)\in \mcV_G(X\da_{E_0}\ua^G)\subseteq \mcV_G(X) =\{\fp\}.
\]
Hence $\res^{*}_{G,E_0}(\fq_0)=\fp$, so that $E_0 \ge E$ and
$\fq_0=\res_{E_0,E}^*(\fq)$ for some pair $(E,\fq)$ corresponding to
$\fp$ as in Theorem~\ref{th:quillen}. Thus $\fq\in\mcV_E(X\da_E)$, by
Theorem~\ref{th:subgp-elem}.

By Theorem~\ref{th:quillen} all pairs $(E,\fq)$ where $\fp$ originates
in $E$ are conjugate, hence if we choose one, then each non-zero
$X\in\Gamma_\fp\sfK(kG)$ has $\Gamma_\fq X\da_E\ne 0$.

Now let $Y$ be another non-zero object in $\Gamma_\fp\sfK(kG)$ and set $Z$ to be the injective resolution of $k(G/E)$, the permutation module. As noted in the proof of Proposition~\ref{pr:V-ind}, the complexes $X\da_{E}\ua^{G}$  and $X\otimes_{k}Z$ are isomorphic in $\KInj{kG}$. This gives the first isomorphism below:
\[
\Hom^{*}_{\sfK(kG)}(X\otimes_k Z,Y) 
	\cong \Hom^{*}_{\sfK(kG)}(X\da_{E}\ua^{G},Y)
	\cong \Hom^{*}_{\sfK(kE)}(X\da_E,Y\da_E).
\]
The second isomorphism is by Frobenius reciprocity. Lemma~\ref{le:kappa-res} implies that $\Gamma_\fq X\da_E$ and $\Gamma_\fq Y\da_E$ are non-zero direct summands of $X\da_E$ and $Y\da_E$, respectively. It follows from Theorem~\ref{th:strat-kE} and the isomorphism above that $\Hom^{*}_{\sfK(kG)}(X\otimes_k Z,Y)$ is non-zero, so $\Gamma_{\fp}\sfK(kG)$ is minimal, by Lemma~\ref{le:loc-minimal}.
\end{proof}

\section{The main theorems}
\label{se:main-theorem}

Let $k$ be a field of characteristic $p$ and $G$ be a finite group,
where $p$ divides the order of $G$. The following result implies
Theorem~\ref{th:loc}.

\begin{theorem}
\label{th:classify-KInj}
The tensor triangulated category $\KInj{kG}$ is stratified by the canonical action of $H^*(G,k)$.
The maps
\[ 
\left\{\begin{gathered}\text{tensor ideal localising}\\
\text{subcategories of $\KInj{kG}$}\end{gathered}\right\}
\xymatrix{\ar@<1ex>[r]^\sigma & \ar@<1ex>[l]^\tau}
\{\text{subsets of $\Spec H^*(G,k)$}\} 
\]
described in \eqref{eq:classification} are mutually inverse bijections.
\end{theorem}

\begin{proof}
This follows from Theorems~\ref{th:transfer-strat} and
\ref{th:KInj-min}.  Note that the support of $\KInj{kG}$ equals $\Spec
H^*(G,k)$ since $\mcV_G(k)=\Spec H^*(G,k)$.
\end{proof}

Next we prove our main result about the stable module category $\StMod(kG)$. We identify it with $\KacInj{kG}$, the full subcategory of $\KInj{kG}$ consisting of acyclic complexes. Observe that $\KacInj{kG}$ is the tensor ideal localising subcategory of $\KInj{kG}$ corresponding to the subset $\Proj H^*(G,k)$ of non-maximal primes; see Proposition \ref{pr:Chou}.  This provides one way of classifying all tensor ideal localising subcategories of $\StMod(kG)$ in terms of subsets of $\Proj H^*(G,k)$; see Proposition~\ref{pr:StMod-to-KInj}.

Now we define an action of the cohomology algebra $H^*(G,k)$ on
$\StMod(kG)$. The equivalence $\KacInj{kG}\xrightarrow\sim\StMod(kG)$
sends a complex $X$ to its cycles $Z^0X$ in degree zero. Restriction
to the category of acyclic complexes induces therefore the following
$H^*(G,k)$-action
\begin{equation*}
\label{eq:ZStmod1}
H^*(G,k)\to Z^*\KInj{kG}\xrightarrow{\mathrm{res}}Z^*\KacInj{kG}\xrightarrow{\sim}Z^*\StMod(kG),
\end{equation*}
where the first map is given by the canonical action of $H^*(G,k)$ on
$\KInj{kG}$. For each $kG$-module $M$, this action induces the map
\begin{equation}
\label{eq:ZStmod} 
H^*(G,k)\hookrightarrow \widehat\Ext^*_{kG}(k,k)=
\End^*_{\StMod(kG)}(k) \xrightarrow{M \otimes_k -}
\End^*_{\StMod(kG)}(M).
\end{equation}

Thus we are in the setting of \S\ref{se:strat}. The support of any module $M$ in $\StMod(kG)$ is denoted $\mcV_G(M)$; it is a subset of $\Spec H^*(G,k)$.  By construction, this coincides with the support of a Tate resolution of $M$ in $\KInj{kG}$. Moreover, $\mcV_{G}(M)$ is the support defined in \cite{Benson/Carlson/Rickard:1996a}. Indeed if $\fp$ is a non-maximal prime $\fp$ in $H^*(G,k)$, then $\Gamma_\fp k$ equals the kappa module for $\fp$ defined in \cite{Benson/Carlson/Rickard:1996a}; see \cite[\S10]{Benson/Iyengar/Krause:2008a}.

\begin{theorem}
\label{th:StMod}
The tensor triangulated category $\StMod(kG)$ is stratified by the action of
$H^*(G,k)$ given by \eqref{eq:ZStmod}. The maps
\[ 
\left\{\begin{gathered}\text{tensor ideal localising}\\
\text{subcategories of $\StMod(kG)$}\end{gathered}\right\}
\xymatrix{\ar@<1ex>[r]^\sigma & \ar@<1ex>[l]^\tau}
\{\text{subsets of $\Proj H^*(G,k)$}\} 
\]
described in \eqref{eq:classification} are mutually inverse bijections.
\end{theorem}

\begin{proof}
Using the identification of $\StMod(kG)$ with $\KacInj{kG}$, the assertion follows from Theorem~\ref{th:classify-KInj} and Proposition~\ref{pr:StMod-to-KInj}.
\end{proof}

The result below implies Theorem~\ref{th:main} from the introduction.

\begin{theorem}
\label{th:Mod}
The maps
\[ 
\left\{\begin{gathered}\text{non-zero tensor ideal localising}\\
\text{subcategories of $\Mod{kG}$}\end{gathered}\right\}
\xymatrix{\ar@<1ex>[r]^\sigma & \ar@<1ex>[l]^\tau}
\{\text{subsets of $\Proj H^*(G,k)$}\} 
\]
given by the obvious analogue of \eqref{eq:classification} are mutually inverse bijections.
\end{theorem}

\begin{proof}
This follows from Theorem~\ref{th:StMod} and Proposition~\ref{pr:Mod-to-StMod}.
\end{proof}

In \cite{Benson/Krause:2008a}, it was proved that if $G$ is a finite $p$-group then there is an equivalence of triangulated categories
\[ 
\KInj{kG} \xrightarrow\sim \sfD(C^*(BG;k))
\]
where $C^*(BG;k)$ denotes the \dg algebra of cochains on the classifying space $BG$ of $G$. This gives $\sfD(C^*(BG;k))$ the structure of a tensor triangulated category. Composing with the canonical map one gets an action
\begin{equation}
\label{eq:ZDdgCBGk}
 H^*(G,k) \to Z^*\KInj{kG} \xrightarrow\sim Z^*\sfD(C^*(BG;k)).
\end{equation}

\begin{theorem}
If $G$ is a finite $p$-group then the tensor triangulated category
$\sfD(C^*(BG;k))$ is stratified by the action of $H^*(G,k)$ given by
\eqref{eq:ZDdgCBGk}.  The maps
\[ 
\left\{\text{localising subcategories of $\sfD(C^*(BG;k))$}\right\}
\xymatrix{\ar@<1ex>[r]^\sigma & \ar@<1ex>[l]^\tau}
\{\text{subsets of $\Spec H^*(G,k)$}\} 
\]
described in \eqref{eq:classification} are mutually inverse bijections.
\end{theorem}
\begin{proof}
This follows from Theorem~\ref{th:classify-KInj} and the observation that every
localising subcategory is tensor ideal.
\end{proof}

\section{Applications}
\label{se:applics}

In this section we deduce the principal theorems of Benson, Carlson,
and Rickard
\cite{Benson/Carlson/Rickard:1996a,Benson/Carlson/Rickard:1997a} from
Theorem \ref{th:classify-KInj} without using shifted subgroups, any
form of Dade's lemma, or algebraic closure of the field.  Then we make
various other deductions from the main theorem. We classify localising
subcategories closed under products, and show that these are the same
as those closed under Brown-Comenetz duality.  We classify the
smashing subcategories, and show that the telescope conjecture holds
for $\StMod(kG)$ and $\KInj{kG}$. Finally, we find the left
perpendicular categories to localising subcategories. Note that
similar applications can be formulated for \dg modules over graded
polynomial algebras and graded exterior algebras; this is left to the
interested reader; see \cite{Benson/Iyengar/Krause:bik2}.

Throughout this section, we abbreviate $\Hom_{\KInj{kG}}(-,-)$ to $\Hom_{\sfK(kG)}(-,-)$.

\subsection*{The tensor product theorem}\ \medskip

\noindent
Part (2) of the result below was proved by Benson, Carlson, and
Rickard \cite{Benson/Carlson/Rickard:1996a} under the additional
hypothesis that $k$ is algebraically closed.

\begin{theorem}
Let $G$ be a finite group and $k$ be a field of characteristic $p$.
\begin{enumerate}[\quad\rm(1)]
\item If $X,Y$ are objects in $\KInj{kG}$, then $\mcV_G(X \otimes_k Y)=\mcV_G(X)\cap\mcV_G(Y)$.
\item If $M,N$ are objects in $\StMod(kG)$, then $\mcV_G(M\otimes_k N)=\mcV_G(M)\cap\mcV_G(N)$.
\end{enumerate}
\end{theorem}
\begin{proof}
Part (2) follows from (1) via the identification of $\StMod{kG}$ with $\KacInj{kG}$.
 
(1) One has an isomorphism $\Gamma_{\fp}(X\otimes_{k}Y)\cong \Gamma_{\fp}X\otimes_{k}\Gamma_{\fp}Y$, which yields an inclusion $\mcV_G(X \otimes_k Y)\subseteq\mcV_G(X)\cap\mcV_G(Y)$.  Conversely, if $\fp\in\mcV_G(X)\cap\mcV_G(Y)$ then $\Gamma_\fp X\ne 0$ and $\Gamma_\fp Y\ne 0$. Theorem \ref{th:classify-KInj} implies that $\Gamma_\fp ik$ is in $\Loc^\otimes(\Gamma_\fp X)$, hence that $\Gamma_\fp Y$ is in $\Loc^\otimes(\Gamma_{\fp}(X\otimes_k Y))$. Since $\Gamma_\fp Y\ne 0$ it follows that $\Gamma_\fp (X \otimes_k Y) \ne 0$.
\end{proof}

\subsection*{The subgroup theorem}\ \medskip

\noindent
The theorem below strengthens Theorem~\ref{th:subgp-elem} from
elementary abelian $p$-groups to all finite groups. For $\StMod(kG)$
and $k$ algebraically closed it was proved in
\cite{Benson/Carlson/Rickard:1996a}.

\begin{theorem}
\label{th:subgp-full}
Let $H\le G$ be a subgroup. For any complex $X$ in $\KInj{kG}$ there is an equality
\[ 
\mcV_{H}(X\da_{H})=(\res_{G,H}^*)^{-1}\mcV_G(X). 
\]
The analogous statement where $\KInj{kG}$ is replaced by $\StMod(kG)$ also holds.
\end{theorem}
\begin{proof}
The argument for $\KInj{kG}$ is the same as the proof of Theorem~\ref{th:subgp-elem} except that one uses Theorem~\ref{th:classify-KInj} instead of Theorem~\ref{th:strat-kE}. The statement about $\StMod(kG)$ is deduced by identifying  it with the acyclic complexes in $\KInj{kG}$.
\end{proof}

\subsection*{The thick subcategory theorem}\ \medskip

\noindent
Recall from \cite[Proposition 2.3]{Krause:2005a} that the compact objects in $\KInj{kG}$ are the semi-injective resolutions of bounded complexes of finitely generated $kG$-modules, and that the canonical functor $\KInj{kG}\to\sfD(\Mod kG)$ induces an equivalence of triangulated categories $\KInjc{kG}\xrightarrow{\sim}\sfD^{\mathsf b}(\mmod kG)$; we view this as an identification.

\begin{lemma}\label{le:fg}
Let $C$ be a direct sum of the simple $kG$-modules. For any complex $X$ in $\sfD^{\mathsf b}(\mmod kG)$ one has $\mcV_G(X)=\mcV(\fa)$ where $\fa$ is the annihilator of the $H^{*}(G,k)$-module $H^*(G,C \otimes_k X)$.
\end{lemma}
\begin{proof}
This is a restatement of \cite[Theorem 5.5 (1)]{Benson/Iyengar/Krause:2008a} in this context.
\end{proof}

In view of the equivalence of tensor triangulated categories
\[
\sfD^{\mathsf b}(\mmod kG)/\sfD^{\mathsf b}(\proj kG)\xrightarrow\sim\stmod (kG)
\]
given by \cite[Theorem 2.1]{Rickard:1989a}, the following theorem  generalises, with a new proof, a classification of the tensor ideal thick subcategories of $\stmod (kG)$ from  \cite{Benson/Carlson/Rickard:1997a}.

\begin{theorem}
\label{th:thick}
There is a one to one correspondence between tensor
ideal thick subcategories of $\sfD^{\mathsf b}(\mmod kG)$ and specialisation
closed subsets $\mcV$ of $\mcV_G$.  

The thick subcategory corresponding to
$\mcV$ is the full subcategory of objects $X$ in $\sfD^{\mathsf b}(\mmod kG)$ such
that $\mcV_G(X)\subseteq\mcV$.
\end{theorem}

\begin{proof}
Let $\sfC$ be a tensor ideal thick subcategory of
$\sfD^{\mathsf b}(\mmod kG)$, and let $\sfC'$ be the tensor ideal
localising subcategory of $\KInj{kG}$ generated by $\sfC$. It
follows using the arguments described in \cite[\S5]{Rickard:1997a}
that $\sfC' \cap \sfD^{\mathsf b}(\mmod kG)=\sfC$. The supports of
$\sfC$ and $\sfC'$ coincide, so the map sending $\sfC$ to its
support is injective by Theorem~\ref{th:classify-KInj}.

It follows from Lemma~\ref{le:fg} that the support of a tensor ideal
thick subcategory is specialisation closed. On the other hand, if
$\mcV$ is a specialisation closed subset of $\mcV_G$, then the support
of the tensor ideal thick subcategory generated by $\{\kos
{ik}{\fp}\mid\fp\in\mcV\}$ equals $\mcV$, by
Proposition~\ref{pr:loc-pr}(2).  So every specialisation closed subset
of $\mcV_G$ occurs as the support of some tensor ideal thick
subcategory $\sfC$ of $\sfD^{\mathsf b}(\mmod kG)$.
\end{proof}

\subsection*{Localising subcategories closed under products and duality}\ \medskip

\noindent
Let $k$ be a field and $(\sfT,\otimes,\one)$ a $k$-linear tensor triangulated category.  There are two notions of duality in $\sfT$. The \emph{Spanier-Whitehead dual} $X^\vee$ of an object $X$ is defined as the function object $\funct(X,\one)$ as in \cite[\S8]{Benson/Iyengar/Krause:2008a} by the adjunction
\[
\Hom_\sfT(-\otimes X,\one)\cong\Hom_\sfT(-,\funct(X,\one)).
\]
The \emph{Brown-Comenetz dual} $X^*$ of an object $X$ is defined  by the isomorphism
\[
\Hom_k(\Hom_\sfT(\one,-\otimes X),k)\cong\Hom_\sfT(-,X^*).
\]
The language and notation is borrowed from stable homotopy theory. The commutativity of the tensor product implies that both dualities are self adjoint. Thus we have a natural isomorphism
$\Hom_\sfT(X,Y^*)\cong \Hom_\sfT(Y,X^*)$ which gives rise to a natural
\emph{biduality morphism} $ X\to X^{**}$.

\begin{lemma}
\label{lem:bc}
The following statements hold for each $X$ in $\sfT$.
\begin{enumerate}[\quad\rm(1)]
\item There is a natural isomorphism $X^*\cong\funct(X,\one^*)$.
\item For each compact object $C$, applying $\Hom_{\sfT}(C,-)$ to the morphism $X\to X^{**}$ yields the biduality map
\[
\Hom_{\sfT}(C,X)\to \Hom_{k}(\Hom_{k}(\Hom_{\sfT}(C,X),k),k).
\]
\item If $X^{*}=0$, then $X=0$.
\end{enumerate}
\end{lemma}

\begin{proof}
The first claim is a consequence of the isomorphisms
\begin{align*}
\Hom_\sfT(-,X^*)&\cong\Hom_\sfT(-\otimes\one,X^*)\\
&\cong\Hom_\sfT(-,\funct(\one,X^*))\\
&\cong\Hom_\sfT(-,\funct(X,\one^*))
\end{align*}
where the last one is immediate from the definition of the Brown-Comenetz dual.

(2) Set $(-)^\dagger=\Hom_k(-,k)$.  There is a natural isomorphism
$\eta\colon\Hom_\sfT(C,X)^\dagger\xra{\sim} \Hom_\sfT(C^\vee,X^*)$
which is compatible with the adjunction isomorphisms for $(-)^*$ and
$(-)^\dagger$. This follows from the defining isomorphism of the
Brown-Comenetz dual and the fact that $C$ is a strongly dualising
object.  Thus the following diagram commutes.
\[
\xymatrix{
\Hom_\sfT(X^*,X^*)\ar[dd]_\sim\ar[rr]^-{\Hom_\sfT(C^\vee,-)}&&
\Hom_k(\Hom_\sfT(C^\vee,X^*),\Hom_\sfT(C^\vee,X^*))\ar[d]_\sim^{\Hom_k(\eta,\eta^{-1})}\\
&&\Hom_k(\Hom_\sfT(C,X)^\dagger,\Hom_\sfT(C,X)^\dagger)\ar[d]_\sim\\
\Hom_\sfT(X,X^{**})\ar[rr]^-{\Hom_\sfT(C,-)}&&
\Hom_k(\Hom_\sfT(C,X),\Hom_\sfT(C,X)^{\dagger\dagger})
}
\]
This justifies the claim.

(3) When $X^{*}=0$ it follows from (2) that $\Hom_{\sfT}(C,X)=0$ for any compact object $C$, since $k$ is a field. Thus $X=0$ as claimed.
\end{proof}

In $\StMod{kG}$ the function object is $\Hom_{k}(M,N)$, with diagonal action. The Spanier-Whitehead dual and the Brown-Comenetz dual of $kG$-modules are closely related. As usual for a $kG$-module $N$ we write $\Omega N$ and $\Omega^{-1}N$ for the kernel of a projective cover of $N$ and the cokernel of an injective envelope of $N$, respectively.

\begin{proposition}
\label{pr:stmodduals}
In $\StMod(kG)$ for each $kG$-module $M$ there are isomorphisms
\[
M^\vee\cong\Hom_k(M,k)\quad\text{and}\quad
M^*\cong\Omega\Hom_k(M,k).
\]
Hence $M^\vee=0$ if and only if $M^{*}=0$ if and only if $M$ is projective.
\end{proposition}

\begin{proof}
The expression for $M^{\vee}$ is by definition. Set $\sfT=\StMod(kG)$. For each $kG$-module $N$, Tate duality~\cite[Chapter XII, Theorem 6.4]{Cartan/Eilenberg:1956} gives the third isomorphism below:
\begin{align*}
\Hom_{\sfT}(N,\Omega\Hom_{k}(M,k)) 
  &\cong \Hom_{\sfT}(\Omega^{-1}N,\Hom_{k}(M,k)) \\
  &\cong \Hom_{\sfT}(\Omega^{-1}N\otimes_{k}M,k) \\
  &\cong \Hom_{k}(\Hom_{\sfT}(k,\Omega((\Omega^{-1}N)\otimes_{k}M)),k) \\
  &\cong \Hom_{k}(\Hom_{\sfT}(k,N\otimes_{k}M),k).
\end{align*}
The other isomorphisms are standard. Thus $M^{*}\cong \Omega\Hom_{k}(M,k)$.

In $\StMod(kG)$ one has $\Omega N=0$ if and only if $N=0$. Therefore the last claim follows from Lemma~\ref{lem:bc}(3).
\end{proof}

The situation in $\KInj{kG}$ is more complicated. Observe that in this category the function object of complexes $X$ and $Y$ is the complex $\Hom_k(X,Y)$ of injective $kG$-modules with diagonal action: $(g\phi)(x)=g(\phi(g^{-1}x))$. 

\begin{proposition}
\label{pr:duals}
For each $X$ in $\KInj{kG}$ there are isomorphisms
\[
X^\vee=\Hom_k(X,ik)\quad\text{and}\quad
X^*\cong\Hom_k(X,ik^*)\cong\Hom_k(X,pk).
\]
Hence $X^{\vee}$ is semi-injective, and $X^\vee=0$ if and only if $X$ is acyclic. Furthermore $X^{*}=0$ if and only if $X=0$.
\end{proposition}

\begin{proof}
The expression for $X^{\vee}$ is by definition. For $X^{*}$ use Lemma~\ref{lem:bc}(1) and an isomorphism $ik^*\cong pk$, which is a variant of Tate duality:
\begin{align*}
\Hom_{\sfK(kG)}(-,ik^*)&\cong \Hom_k(\Hom_{\sfK(kG)}(ik,-),k)\\
&\cong\Hom_k(\Hom_{\sfK(kG)}(k,-),k)\\
&\cong\Hom_{\sfK(kG)}(-,k)\\
&\cong\Hom_{\sfK(kG)}(-,pk).
\end{align*}
The  adjunction isomorphism $\Hom_{\sfT}(-,X^{\vee})\cong \Hom_{\sfT}(X\otimes_{k}-,\ik)$ implies that $X^{\vee}$ is semi-injective, since $ik$ is semi-injective. Given this it is clear that $X^{\vee}=0$ precisely when $X$ is acyclic. The statement about $X^{*}$ is part of Lemma~\ref{lem:bc}.
\end{proof}

Given a subcategory $\sfC$ of a triangulated category $\sfT$ we define full subcategories
\begin{align*}
{^\perp}\sfC&=\{X\in\sfT\mid \text{$\Hom_\sfT^*(X,Y)=0$ for all $Y$ in $\sfC$}\} \\
\sfC^{\perp}&=\{X\in\sfT\mid \text{$\Hom_\sfT^*(Y,X)=0$ for all $Y$ in $\sfC$}\}.
\end{align*}
Evidently, ${^\perp}\sfC$ is a localising subcategory, and $\sfC^\perp$ is a colocalising subcategory, i.e., a thick subcategory closed under products.

\begin{theorem}
\label{th:dual-prod}
For any tensor ideal localising subcategory $\sfC$ of $\KInj{kG}$ the
following are equivalent:
\begin{enumerate}[\quad\rm(a)]
\item The subcategory $\sfC$ is closed under products.
\item The complement of the support of $\sfC$ in $\Spec H^*(G,k)$ is specialisation closed.
\item The subcategory $\sfC$ is equal to $\sfD^\perp$ with $\sfD$ a subcategory of compact objects.
\item The Brown-Comenetz dual of any object in $\sfC$ is also in $\sfC$.
\end{enumerate}
The analogous statement where $\KInj{kG}$ is replaced by $\StMod(kG)$ and the set $\Spec H^*(G,k)$ is replaced by  $\Proj H^*(G,k)$ also holds.
\end{theorem}

The proof of this result relies on a construction of objects $T(I)$ in $\KInj{kG}$ which we recall from
\cite[\S11]{Benson/Krause:2008a}. Given an injective $H^*(G,k)$-module $I$, the object $T(I)$ is defined in terms of the following natural isomorphism
\begin{equation}\label{eq:TI}
\Hom_{H^*(G,k)}(H^*(G,-),I)\cong\Hom_{\sfK(kG)}(-,T(I)).
\end{equation} 
For each prime ideal $\fp$ of $H^*(G,k)$ let $I(\fp)$ be the injective envelope of $H^*(G,k)/\fp$.

\begin{lemma}
\label{le:TI}
Let $\fp$ and $\fq$ be prime ideals in $H^*(G,k)$ with  $\fq\subseteq \fp$, and let $I$ be an injective
$H^*(G,k)$-module.
\begin{enumerate}[\quad\rm (1)]
\item $\mcV_G(T(I(\fp)))=\{\fp\}$.
\item $T(I(\fq))$ is a direct summand of a direct product of shifts of
$T(I(\fp))$.
\item The natural morphism $T(I)\to T(I)^{**}$ is a split monomorphism.
\end{enumerate}
\end{lemma}

\begin{proof}
(1) Let $C$ be a compact object of $\KInj{kG}$.  Then $H^{*}(G,C)$ is
finitely generated as a module over $H^{*}(G,k)$. It follows that
$\Hom^{*}_{\sfK(kG)}(C,T(I(\fp)))$ is $\fp$-local and $\fp$-torsion as a $H^{*}(G,k)$-module,
for it is isomorphic to
\[
\Hom^{*}_{H^*(G,k)}(H^*(G,C),I(\fp)).
\]
Now apply \cite[Corollary 5.9]{Benson/Iyengar/Krause:2008a}.

(2) The module $I(\fq)$ is $\fp$-local and the shifted copies of $I(\fp)$
form a set of injective cogenerators for the category of $\fp$-local
modules.  Thus $I(\fq)$ is a direct summand of a product of shifted
copies of $I(\fp)$.  Now apply the functor $T$ and observe that $T$
preserves products.

(3) The induced map $H^*(G,T(I))\to H^*(G,T(I)^{**})$ is a monomorphism, by Lemma~\ref{lem:bc}(2). Applying $\Hom_{H^*(G,k)}(-,I)$ to this map and using \eqref{eq:TI}, one gets an epimorphism
\[
\Hom_{\sfK(kG)}(T(I)^{**},T(I))\to\Hom_{\sfK(kG)}(T(I),T(I))
\]
which provides an inverse for $T(I)\to T(I)^{**}$.
\end{proof}

\begin{proof}[Proof of Theorem \emph{\ref{th:dual-prod}}]
First we prove the theorem for $\KInj{kG}$.

(a) $\Rightarrow$ (b):  Let $\fp$ be a prime ideal in the support of $\sfC$. Theorem~\ref{th:classify-KInj} implies that $T(I(\fp))$ is in $\sfC$, since its support is $\{\fp\}$, by Lemma~\ref{le:TI}(1).
Therefore $T(I(\fq))$ is in $\sfC$ for every $\fq\subseteq \fp$ by Lemma~\ref{le:TI}, since $\sfC$ is closed under products. Thus the complement of the support of $\sfC$ is specialisation closed.

(b) $\Rightarrow$ (c): Let $\mcV$ denote the complement of the support
of $\sfC$. Since it is specialisation closed the localising
subcategory $\sfK_\mcV$ of $\KInj{kG}$ corresponding to $\mcV$ is
generated by compact objects, by Proposition~\ref{pr:loc-pr}.
Therefore $\sfK_\mcV^\perp=(\sfK^{\mathsf c}_\mcV)^\perp$ where $\sfK^{\mathsf c}_\mcV$
denotes the full subcategory consisting of the compact objects in
$\sfK_\mcV$. On the other hand, $\sfK_\mcV^\perp$ is the localising
subcategory consisting of all objects with support contained in the
complement of $\mcV$, by \cite[Corollary
5.7]{Benson/Iyengar/Krause:2008a}. Thus Theorem \ref{th:classify-KInj}
implies $\sfC=(\sfK^{\mathsf c}_\mcV)^\perp$.

(c) $\Rightarrow$ (a): This implication is clear.

(c) $\Rightarrow$ (d): Suppose that $\sfC=\sfD^\perp$ with $\sfD$ a subcategory of compact objects. Fix objects $D$ in $\sfD$ and $X$ in $\sfC$. We need to show that  $\Hom_{\sfK(kG)}(D,X^*)=0$.  
For any compact object $C$ there are isomorphisms
\begin{align*}
\Hom_{\sfK(kG)}(D\otimes_k C,X)&\cong\Hom_{\sfK(kG)}(D,\Hom_k(C,X))\\
&\cong\Hom_{\sfK(kG)}(D,C^\vee\otimes_k X)=0
\end{align*} 
where the last one holds because $\sfC$ is tensor ideal. Hence
$D\otimes_{k}C$ is in ${}^{\perp}\sfC$ for any compact object
$C$. Thus $D^\vee\otimes_{k}C$ is in ${}^{\perp}\sfC$ as well, since
$D^{\vee}$ is a direct summand of
$D^{\vee}\otimes_{k}D\otimes_{k}D^{\vee}$. A similar argument now
yields
\[
\Hom_{\sfK(kG)}(C,D\otimes_{k}X)\cong\Hom_{\sfK(kG)}(C\otimes_k D^{\vee},X)=0
\]
for any compact object $C$. Therefore $D\otimes_k X=0$, so that
\begin{align*}
\Hom_{\sfK(kG)}(D,X^*)&\cong\Hom_{\sfK(kG)}(D,\Hom_k(X,ik^*))\\
&\cong\Hom_{\sfK(kG)}(D\otimes_k X,ik^*)=0.
\end{align*}

(d) $\Rightarrow$ (b): Suppose that $\sfC$ is closed under
Brown-Comenetz duality, and let $\fp$ be a prime in the support of
$\sfC$. We apply Lemma~\ref{le:TI} several times. First notice that
$T(I(\fp))$ is in $\sfC$, by Theorem \ref{th:classify-KInj}.  Given
any set $\{X_\alpha\}$ where each $X_\alpha$ is a shift of
$T(I(\fp))$, it follows that the complex
\[ 
\prod_{\alpha}X_\alpha^{**} = 
\Bigl(\bigoplus_{\alpha}X_\alpha^*\Bigr)^* 
\] 
is in $\sfC$. The natural map $\prod_{\alpha} X_\alpha \to
\prod_{\alpha}X_\alpha^{**}$ is a split monomorphism, and hence
$\prod_{\alpha}X_\alpha$ is in $\sfC$. If $\fq\subseteq\fp$ then
$T(I(\fq))$ is a direct summand of a direct product of shifts of
$T(I(\fp))$. It follows that $\fq$ is also in the support of $\sfC$.
Thus the complement of this set is specialisation closed.

This completes the proof of the theorem for $\KInj{kG}$. It remains
to consider the category $\StMod(kG)$. We use the same arguments as
before for the implications (b) $\Rightarrow$ (c) $\Rightarrow$ (d)
and (c) $\Rightarrow$ (a), because they are formal, given
Theorem~\ref{th:StMod}. For the other implications, we identify
$\StMod(kG)$ with the category $\KacInj{kG}$ of acyclic complexes and
view each tensor ideal localising subcategory of $\StMod(kG)$ as a
tensor ideal localising subcategory of $\KInj{kG}$. Then we use the
fact that the inclusion into $\KInj{kG}$ preserves products and
Brown-Comenetz duals. Moreover, the support of $\KacInj{kG}$
equals $\Proj H^*(G,k)$, by Proposition~\ref{pr:Chou}.
\end{proof}

\subsection*{The telescope conjecture for $\StMod(kG)$ and 
$\KInj{kG}$}\ \medskip

\noindent
A localising subcategory $\sfC$ of a triangulated category $\sfT$ is
\emph{strictly localising} if the inclusion has a right adjoint.  This
is equivalent to the statement that there is a localisation functor
$L\col\sfT\to\sfT$ such that an object $X$ of $\sfT$ is in $\sfC$ if
and only if $LX=0$; see for example \cite[Lemma~3.5]{Krause:2000a}.
The obstruction to constructing the right adjoint is that the
collections of morphisms in the Verdier quotient may be too large to
be sets.

\begin{lemma}\label{le:strict}
Tensor ideal localising subcategories of $\KInj{kG}$ and $\StMod(kG)$ are strictly localising.
\end{lemma}
\begin{proof}
Let $\sfC$ be a tensor ideal localising subcategory of $\KInj{kG}$,
and let $\mcV$ be its support, which is a subset of $\Spec
H^*(G,k)$. Fix a compact generator $C$ of $\KInj{kG}$ and consider the
functor from $\KInj{kG}$ to the category of graded abelian groups
sending an object $X$ to
$\bigoplus_{\fp\not\in\mcV}\Hom^*_{\sfK(kG)}(C,\Gamma_\fp X)$.
Theorem~\ref{th:classify-KInj} implies that its kernel is $\sfC$. The
functor is cohomological and preserves coproducts. Thus there
exists a localisation functor $\KInj{kG}\to\KInj{kG}$ with kernel
$\sfC$; see for example \cite[Proposition
3.6]{Benson/Iyengar/Krause:2008a}.

An analogous argument works for $\StMod(kG)$.
\end{proof}

A strictly localising subcategory $\sfC$ of a triangulated category
$\sfT$ is \emph{smashing} if the localisation functor $\sfT\to\sfT$
with kernel $\sfC$ preserves coproducts. If $\sfT$ is tensor
triangulated and generated by its tensor unit $\one$, then a
localisation functor $L\colon\sfT\to\sfT$ preserves coproducts if and
only if the natural morphism $L\one\otimes X\to LX$ is an isomorphism
for all $X$ in $\sfT$. This fact explains the term ``smashing'',
because in algebraic topology the smash product plays the role of the
tensor product. 

Next we discuss the telescope conjecture which is due to Bousfield and
Ravenel \cite{Bousfield:1979a,Ravenel:1984a}. In its general form, the
conjecture asserts for a compactly generated triangulated category
$\sfT$ that every smashing localising subcategory is generated by
objects that are compact in $\sfT$; see \cite{Neeman:1992b}. The
following result confirms this conjecture for $\KInj{kG}$ and
$\StMod(kG)$, at least for all smashing subcategories that are tensor
ideal.

\begin{theorem}
\label{th:tel}
Let $\sfC$ be a tensor ideal localising subcategory of $\KInj{kG}$. 
The following conditions are equivalent:
\begin{enumerate}[\quad\rm (a)]
\item The localising subcategory $\sfC$ is smashing.
\item The localising subcategory $\sfC$ is generated by objects
compact in $\KInj{kG}$.
\item The support of $\sfC$  is a specialisation closed subset of $\Spec H^*(G,k)$.
\end{enumerate}
A similar result holds for $\StMod(kG)$ with $\Spec H^*(G,k)$ replaced by  $\Proj H^*(G,k)$.
\end{theorem}
\begin{proof}
We prove the theorem for $\KInj{kG}$; an analogous argument works for
the stable module category.  Let $L\col\KInj{kG}\to\KInj{kG}$ be a
localisation functor with kernel $\sfC$.  For each object $X$ in
$\KInj{kG}$, there exists a localisation triangle $\Gamma X\to X\to
LX\to $ with $\Gamma X$ in $\sfC$ and $LX$ in $\sfC^\perp$.  Using
this exact triangle one shows that $\sfC$ is smashing if and only if
$\sfC^\perp$ is closed under coproducts.

(a) $\Rightarrow$ (b): If $\sfC$ is smashing then $\sfC^\perp$ is a localising subcategory closed under products. Thus $\sfC^\perp=\sfD^\perp$ for some category $\sfD$ consisting of compact objects by Theorem~\ref{th:dual-prod}.  It follows that $\sfC$ is generated by $\sfD$.

(b) $\Rightarrow$ (c): Let $\sfD$ be a subcategory of compact objects
such that $\sfC=\Loc(\sfD)$. Then $\sfC$ and $\sfD$ have same support.
Now one uses that the support of each compact object is specialisation
closed by Lemma~\ref{le:fg}.

(c) $\Rightarrow$ (a): Let $\mcV$ be the support of $\sfC$. Then
$\sfC$ consists of all objects with support contained in $\mcV$, by
Theorem~\ref{th:classify-KInj}.  This implies that $\sfC^\perp$ is the
localising subcategory consisting of all objects with support disjoint
from $\mcV$, since $\mcV$ is specialisation closed; see
\cite[Corollary 5.7]{Benson/Iyengar/Krause:2008a}. In particular,
$\sfC^\perp$ is closed under coproducts, and therefore $\sfC$ is
smashing.
\end{proof}

\subsection*{Left perpendicular categories}\ \medskip

\noindent
If $\mcV$ is a subset of $\Spec H^*(G,k)$, we write $\cl(\mcV)$ for the specialisation closure of $\mcV$, namely the smallest specialisation closed subset containing it.

\begin{theorem}
\label{th:zero-hom}
For $X$ and $Y$ in $\KInj{kG}$ the following are equivalent:
\begin{enumerate}[\quad\rm (a)]
\item $\Hom^*_{\sfK(kG)}(X,Y')=0$ for all $Y'\in\Loc^\otimes(Y)$.
\item $\cl(\mcV_G(X))\cap\mcV_G(Y)=\varnothing$
\end{enumerate}
\end{theorem}

\begin{proof}
The implication (b) $\Rightarrow$ (a) is part of 
\cite[Corollary 5.8]{Benson/Iyengar/Krause:2008a}.  
Assume (a) holds. Choose primes
$\fq\in\mcV_G(Y)$ and $\fp\in\mcV_G(X)$, and a compact object $C$ in
$\KInj{kG}$ with $\Hom_k(C,\Gamma_\fp X)\ne 0$. Since
$\mcV_G(T(I(\fq)))=\{\fq\}$, by Lemma~\ref{le:TI},
Theorem~\ref{th:classify-KInj} yields
$C\otimes_{k}T(I(\fq))\in\Loc^\otimes(Y)$.  Since $\Gamma_\fp X =
\Gamma_\fp ik \otimes_k X \in \Loc^\otimes(X)$ holds, one has
\begin{align*} 
\Hom_{H^*(G,k)}^*(H^*(G,\Hom_k(C,\Gamma_\fp X)),I(\fq))&\cong
\Hom^*_{\sfK(kG)}(\Hom_k(C,\Gamma_{\fp}X),T(I(\fq)))\\
&\cong \Hom^*_{\sfK(kG)}(\Gamma_\fp X,C\otimes_k T(I(\fq)))=0.
\end{align*}
The $H^*(G,k)$-module $H^*(G,\Hom_k(C,\Gamma_\fp X))$ is
non-zero and $\fp$-local, so it follows that $\fp\not\subseteq\fq$
as required.
\end{proof}

If $\mcV$ is a subset of $\mcV_G$, we write ${^\perp}\mcV$ for
the set of primes $\fq\in\mcV_G$ such that for all $\fp\in\mcV$
we have $\fq\not\supseteq\fp$. In other words, ${^\perp}\mcV$
is the largest specialisation closed subset of $\mcV_G$ that
has trivial intersection with $\mcV$. The statement of the result below 
was suggested by a question of Jeremy Rickard.

\begin{corollary}
Let $\sfC$ be a tensor ideal localising subcategory of $\KInj{kG}$. If
$\mcV$ is the support of $\sfC$, then ${^\perp}\mcV$ is the support of
${^\perp}\sfC$.
\end{corollary}
\begin{proof}
This follows from Theorem~\ref{th:zero-hom}.
\end{proof}


\begin{thebibliography}{10}

\bibitem{Avramov/Buchweitz/Iyengar/Miller:hffc}
L.~L. Avramov, R.-O. Buchweitz, S.~B. Iyengar, and C.~Miller, \emph{{Homology
  of perfect complexes}}, Adv.\ Math. \textbf{223} (2010), 1731--1781; Corrigendum: 
  Adv.\ Math. \textbf{225} (2010) 3576--3578.
  
\bibitem{Avramov/Foxby/Halperin:dga}
L.~L. Avramov, H.-B. Foxby, and S.~Halperin, \emph{Differential graded homological algebra}, Preliminary version, 2008.

\bibitem{Benson:1991a}
D.~J. Benson, \emph{{Representations and Cohomology I: Basic representation
  theory of finite groups and associative algebras}}, Cambridge Studies in
  Advanced Mathematics, vol.~30, Cambridge University Press, 1991, reprinted in
  paperback, 1998.

\bibitem{Benson/Carlson/Rickard:1996a}
D.~J. Benson, J.~F. Carlson, and J.~Rickard, \emph{{Complexity and varieties
  for infinitely generated modules, II}}, Math.\ Proc.\ Camb.\ Phil.\ Soc.
  \textbf{120} (1996), 597--615.

\bibitem{Benson/Carlson/Rickard:1997a}
D.~J. Benson, J.~F. Carlson, and J.~Rickard, \emph{{Thick subcategories of the stable module category}}, Fundamenta
  Mathematicae \textbf{153} (1997), 59--80.

\bibitem{Benson/Iyengar/Krause:2008a}
D.~J. Benson, S.~B. Iyengar, and H.~Krause, \emph{{Local cohomology and support for triangulated categories}},
  Ann.\ Scient.\ \'Ec.\ Norm.\ Sup.\ (4) \textbf{41} (2008), 1--47.

\bibitem{Benson/Iyengar/Krause:bik2}
D.~J. Benson, S.~B. Iyengar, and H.~Krause, \emph{{Stratifying triangulated
  categories}}, J.~Topology (to appear) \texttt{arXiv:0910.0642}

\bibitem{Benson/Krause:2008a}
D.~J. Benson and H.~Krause, \emph{{Complexes of injective $kG$-modules}},
  Algebra \& Number Theory \textbf{2} (2008), 1--30.

\bibitem{Bousfield:1979a}
A.~K. Bousfield, \emph{{The localization of spectra with respect to homology}},
  Topology \textbf{18} (1979), 257--281.

\bibitem{Bruns/Herzog:1998a}
W.~Bruns and J.~Herzog, \emph{{Cohen--Macaulay rings}}, Cambridge Studies in
  Advanced Mathematics, vol.~39, Cambridge University Press, 1998, Revised
  edition.

\bibitem{Carlson:1996a}
J.~F. Carlson, \emph{{Modules and Group Algebras}}, Lectures in Mathematics,
  ETH Z\"urich, Birkh\"auser, 1996.

\bibitem{Cartan/Eilenberg:1956}
H.~Cartan and S.~Eilenberg, \emph{Homological Algebra}, Princeton University Press, Princeton, NJ, 1956.

\bibitem{Chouinard:1976a}
L.~Chouinard, \emph{{Projectivity and relative projectivity over group rings}},
  J.~Pure \& Applied Algebra \textbf{7} (1976), 278--302.

\bibitem{Deligne:1977a}
P.~Deligne, \emph{{S\'eminaire de G\'eom\'etrie Alg\'ebrique du Bois-Marie
  SGA4$\frac{1}{2}$}}, Lecture Notes in Mathematics, vol. 569, Springer-Verlag,
  Ber\-lin/New York, 1977.

\bibitem{Evens:1961a}
L.~Evens, \emph{{The cohomology ring of a finite group}}, Trans.\ Amer.\ Math.\
  Soc. \textbf{101} (1961), 224--239.

\bibitem{Felix/Halperin/Thomas:2001a}
Y.~F\'elix, S.~Halperin, and J.-C. Thomas, \emph{{Rational homotopy theory}},
  Graduate Texts in Mathematics, vol. 205, Springer-Verlag, Ber\-lin/New York,
  2001.

\bibitem{Friedlander/Suslin:1997a}
E.~M. Friedlander and A.~Suslin, \emph{{Cohomology of finite group schemes over
  a field}}, Invent.\ Math. \textbf{127} (1997), 209--270.

\bibitem{Golod:1959a}
E.~S. Golod, \emph{{The cohomology ring of a finite $p$-group}}, Dokl.\ Akad.\
  Nauk.\ SSSR \textbf{125} (1959), 703--706, (Russian).

\bibitem{Happel:1988a}
D.~Happel, \emph{{Triangulated categories in the representation theory of
  finite dimensional algebras}}, London Math.\ Soc.\ Lecture Note Series, vol.
  119, Cambridge University Press, 1988.

\bibitem{Hopkins:1987a}
M.~J. Hopkins, \emph{{Global methods in homotopy theory}}, Homotopy Theory,
  Durham 1985, Lecture Notes in Mathematics, vol. 117, Cambridge University
  Press, 1987.

\bibitem{Hovey/Palmieri/Strickland:1997a}
M.~Hovey, J.~H. Palmieri, and N.~P. Strickland, \emph{{Axiomatic stable
  homotopy theory}}, Mem.\ AMS, vol. 128, American Math.\ Society, 1997.

\bibitem{Keller:1994a}
B.~Keller, \emph{{Deriving DG categories}}, Ann.\ Scient.\ \'Ec.\ Norm.\ Sup.\
  (4) \textbf{27} (1994), 63--102.

\bibitem{Krause:2000a}
H.~Krause, \emph{{Smashing subcategories and the telescope conjecture - an
  algebraic approach}}, Invent.\ Math. \textbf{139} (2000), 99--133.

\bibitem{Krause:2005a}
H.~Krause, \emph{{The stable derived category of a noetherian scheme}},
  Compositio Math. \textbf{141} (2005), 1128--1162.

\bibitem{Neeman:1992a}
A.~Neeman, \emph{{The chromatic tower for $D(R)$}}, Topology \textbf{31}
  (1992), 519--532.

\bibitem{Neeman:1992b}
A.~Neeman, \emph{{The connection between the $K$-theory localization theorem of
  Thomason, Trobaugh and Yao and the smashing subcategories of Bousfield and
  Ravenel}}, Ann.\ Scient.\ \'Ec.\ Norm.\ Sup.\ (4) \textbf{25} (1992),
  547--566.

\bibitem{Neeman:2001a}
A.~Neeman, \emph{{Triangulated categories}}, Annals of Math. Studies, vol. 148,
  Princeton Univ.\ Press, 2001.

\bibitem{Quillen:1971b}
D.~G. Quillen, \emph{{The spectrum of an equivariant cohomology ring, I}},
  Ann.\ of Math. \textbf{94} (1971), 549--572.

\bibitem{Quillen:1971c}
D.~G. Quillen, \emph{{The spectrum of an equivariant cohomology ring, II}}, Ann.\ of
  Math. \textbf{94} (1971), 573--602.

\bibitem{Ravenel:1984a}
D.~C. Ravenel, \emph{{Localization with respect to certain periodic homology
  theories}}, Amer.\ J.\ Math. \textbf{106} (1984), 351--414.

\bibitem{Rickard:1989a}
J.~Rickard, \emph{{Derived categories and stable equivalence}}, J.~Pure \&
  Applied Algebra \textbf{61} (1989), 303--317.

\bibitem{Rickard:1997a}
J.~Rickard, \emph{{Idempotent modules in the stable category}}, J.~London Math.\
  Soc. \textbf{178} (1997), 149--170.

\bibitem{Venkov:1959a}
B.~B. Venkov, \emph{{Cohomology algebras for some classifying spaces}}, Dokl.\
  Akad.\ Nauk.\ SSSR \textbf{127} (1959), 943--944.

\bibitem{Verdier:1996a}
J.-L. Verdier, \emph{{Des cat\'egories d\'eriv\'ees des cat\'egories
  ab\'eliennes}}, Ast\'erisque, vol. 239, Soci\'et\'e Math.\ de France, 1996.

\end{thebibliography}
\end{document}